\UseRawInputEncoding

\documentclass[reqno,english]{amsart}
\usepackage{amsfonts,amsmath,latexsym,verbatim,amscd,mathrsfs,color,array}

\usepackage{amsmath,amssymb,amsthm,amsfonts,graphicx,color}
\usepackage{amssymb}
\usepackage{pdfsync}
\usepackage{epstopdf}
\usepackage[colorlinks=true]{hyperref}
\usepackage{subfigure}

\makeatletter
\def\@currentlabel{2.1}\label{e:dispaa}
\def\@currentlabel{2.21}\label{e:dispau}
\def\@currentlabel{2.22}\label{e:dispav}
\def\@currentlabel{2.23}\label{e:dispaw}
\def\@currentlabel{2.24}\label{e:dispax}
\def\theequation{\thesection.\@arabic\c@equation}
\makeatother
\let\oldbibliography\thebibliography
\renewcommand{\thebibliography}[1]{%
\oldbibliography{#1}%
\setlength{\itemsep}{0pt}%
}

\renewcommand{\theequation}{\thesection.\arabic{equation}}
\newtheorem{lemma}{Lemma}[section]

\newtheorem{definition}{Definition}

\newtheorem{proposition}{Proposition}[section]
\newtheorem{corollary}{Corollary}[section]
\newtheorem{remark}{Remark}[section]
\newtheorem{conjecture}{Conjecture}[section]
\newtheorem{open problem}{Open Problem}[section]
\newtheorem{open question}{Open Quesion}[section]
\newcommand{\bremark}{\begin{remark} \em}
\newcommand{\eremark}{\end{remark} }

\newtheorem{numerical/experimental results}{Numerical/Experimental results}[section]

\newtheorem{theorem}{Theorem}[section]

\newcommand{\R}{ {\mathbb R}}
\newcommand{\BE}{\begin{equation}}
\newcommand{\BEN}{\begin{equation*}}
\newcommand{\EE}{\end{equation}}
\newcommand{\EEN}{\end{equation*}}
\newcommand{\BL}{\begin{lemma}}
\newcommand{\EL}{\end{lemma}}
\newcommand{\BT}{\begin{theorem}}
\newcommand{\ET}{\end{theorem}}
\newcommand{\BP}{\begin{proposition}}
\newcommand{\EP}{\end{proposition}}
\newcommand{\BC}{\begin{corollary}}
\newcommand{\EC}{\end{corollary}}
\renewcommand{\Re}{\operatorname{Re}}
\renewcommand{\Im}{\operatorname{Im}}

\begin{document}


\title[Minima of difference of theta functions]{\bf On minima of  difference of theta functions and application to hexagonal crystallization}

\author{Senping Luo}

\author{Juncheng Wei}

\address[S.~Luo]{School of Mathematics and statistics, Jiangxi Normal University, Nanchang, 330022, China}
\address[J.~Wei]{Department of Mathematics, University of British Columbia, Vancouver, B.C., Canada, V6T 1Z2}

\email[S.~Luo]{luosp1989@163.com or spluo@jxnu.edu.cn}

\email[J.~Wei]{jcwei@math.ubc.ca}

\begin{abstract}
Let  $z=x+iy \in \mathbb{H}:=\{z= x+ i y\in\mathbb{C}: y>0\}$  and $
\theta (\alpha;z)=\sum_{(m,n)\in\mathbb{Z}^2 } e^{-\alpha \frac{\pi }{y }|mz+n|^2}$ be the theta function associated with the lattice $L ={\mathbb Z}\oplus z{\mathbb Z}$. In this paper we consider the following minimization problem of difference of two theta functions
\begin{equation}\aligned\nonumber
\min_{   \mathbb{H} } \Big(\theta (\alpha; z)-\beta\theta (2\alpha; z)\Big)
\endaligned\end{equation}
where $\alpha \geq 1$ and $ \beta \in (-\infty, +\infty)$. We prove that there is a critical value $\beta_c=\sqrt2$ (independent of $\alpha$) such that if $\beta\leq\beta_c$, the minimizer  is $\frac{1}{2}+i\frac{\sqrt3}{2}$ (up to translation and rotation) which corresponds to the hexagonal lattice, and if $\beta>\beta_c$, the minimizer does not exist.  Our result  partially answers some questions raised in  \cite{Bet2016, Bet2018, Bet2020, Bet2019AMP} and gives a new proof in   the crystallization of hexagonal
lattice under Yukawa potential.

\end{abstract}

\maketitle


\section{Introduction and Statement of Main Results}
\setcounter{equation}{0}

Let $L$ be a two dimensional lattice. A large class of physical problems can be reduced to the following minimization problem:
 \begin{equation}\aligned\label{EFL}
\min_L E_f(\Lambda ), \;\;\hbox{where}\;\;E_f(L):=\sum_{\mathbb{P}\in L \backslash\{0\}} f(|\mathbb{P}|^2).
\endaligned\end{equation}
The function $f$ denotes the potential of the system and  the summation ranges over all the lattice points except for the origin $0$. The function
$E_f(L)$ denotes the total energy of the system under the background potential $f$ over a periodical lattice $L$, which arises in various physical problems (\cite{Bet2015,Bet2016,Bet2018,Bet2019,Bet2019AMP,Betermin2021JPA, Serfaty2018}). For example there is a clear connection of lattice sum and Abrikosov vortex lattices (see e.g. \cite{Abr}, \cite{Che2007}, \cite{Sigal2018}, \cite{Serfaty2010}, \cite{Serfaty2012}, \cite{Serfaty2014}, \cite{Luo2019}).
In the physical aspect, $E_f(L)$ refers to crystal lattice energy (\cite{ChenPRL1990,Shen2005}) and Hamiltonian of crystals with long-ranged interaction (\cite{MatAMS}). The function $E_f(L)$ has deep link with the partition function which is fundamental in equilibrium statistical physics. Locating the minimizer of such a total energy $E_f(L)$ over all the shapes of the lattices has important applications  in physics (\cite{Betermin2021JPA, Fey, Ho2001,Mue2002, Mat1999}), number theory (see e.g., \cite{Osg1988}, \cite{Sar2006}, \cite{Bla2015}), adsorption on non-ideal surfaces \cite{Cero1993}, etc.  The application of number theory to physics has many aspects and some of them reveal unexpected discovery. Significant examples of this direction went back to \cite{Cohen}, \cite{ChenPRL1990}, \cite{Sch2011}, \cite{Yuan2010} and the references therein. For other examples see  the book \cite{Number}.

Let $ z\in \mathbb{H}:=\{z= x+ i y\in\mathbb{C}: y>0\}$ and $L ={\mathbb Z}\oplus z{\mathbb Z}$ be the lattice in $ \mathbb{R}^2$.
When $ f(\cdot)=e^{- \alpha |\cdot|^2}$ ($\alpha>0$) and the corresponding $E_f (L)$ becomes the the theta function
 \begin{equation}
 \label{thetas}
\theta (\alpha; z):=\sum_{\mathbb{P}\in\Lambda} e^{- \alpha |\mathbb{P}|^2}=\sum_{(m,n)\in\mathbb{Z}^2} e^{-\alpha \frac{\pi }{y }|mz+n|^2},
\end{equation}
a celebrated result of  Montgomery (\cite{Mon1988}) states that the hexagonal lattice attains the minimimizer in (\ref{EFL}).  By the classical Bernstein representation formula, Montgomery's result can be extended to any completely monotone functions, which are   $C^\infty(0,\infty)$ satisfying
\begin{equation}
\label{Complete}
(-1)^j f^{(j)}(x)>0, j=0,1,2,\cdots \infty.
\end{equation}

In  \cite{LW2021} we  have considered the following minimization problem of sum of two theta functions
\begin{equation}\aligned\label{EFL1}
\min_{z \in {\mathbb H}} \Big(\theta(\alpha; z)+\rho \theta (\alpha; \frac{z+1}{2})\Big)
\endaligned\end{equation}
and showed that the hexagonal-rhombic-square-rectangular transition appears as $\rho$ goes from $0$ to $+\infty$. This result can also be generalized to completely monotone functions.

 However, in many physical models, the potential function $ f$ may not be completely monotone. A classical example is the Lennard-Jones potential $ f(r)= \frac{1}{r^6}- \frac{2}{r^3}$ (see e.g. \cite{Baskes}) and its generalization
     $f(r)=\frac{a_1}{r^{t_1}}-\frac{a_2}{r^{t_2}}$, where $t_1>t_2>0, a_1,a_2>0$.    In \cite{Bet2018, Bet2020, Bet2019AMP},  B\'etermin and his collaborators   initiated  a theoretical study of the \eqref{EFL} with the Lennard-Jones type potential.  A numerical simulation suggests the hexagonal-rhombic-square-rectangular lattice phase transitions. Observe that the Lennard-Jones potential consists of two parts, each part being completely monotone. In addition, it is of one-well potential as introduced and defined in \cite{Bet2018}, i.e., there exists $a>0$ such that f is nonincreasing on $(0, a)$ and nondecreasing
on $(a, +\infty)$. The following conjecture was made in \cite{Bet2018}:
     \begin{conjecture}[\cite{Bet2018}, last page; open question 1.16  of \cite{Bet2019AMP}]\label{BetConjecture}
     \begin{itemize}
       \item The behavior of the minimizers of \eqref{EFL} with respect to the lattice area A is qualitatively the
same for all the Lennard-Jones type potentials (i.e., admits hexagonal-rhombic-square-rectangular lattice phase transitions);

       \item more generally, we can imagine that we should find the same result for any potential f written
as $f:=f_1-f_2$,
where $f_1$ and $f_2$ are both completely monotone and f is of one-well.
     \end{itemize}

     \end{conjecture}

For non-monotone potentials, there are other interesting open problems concerning the energy functional \eqref{EFL}, which we list some of them here:

\begin{conjecture}[Conjecture 2.7 of B\'etermin, Faulhuber and Kn$\ddot{u}$pfer \cite{Bet2020}]\label{Conjecture2}
The existence of square lattice minimizes the lattice energy $E_f(L)$ when $f(r)=e^{-\beta \pi r}-e^{-\alpha \pi r}$ for $|\beta-\alpha|$
bigger than some small positive number.
\end{conjecture}

\begin{open question}[B\'etermin-Petrache, \cite{Bet2019AMP}]\label{Open1}
Question 1.1: If $f(r^2)$ is not a positive superposition of Gaussians, can the triangular(hexagonal)
lattice still be a minimizer of $E_f[L]$ among lattices at any fixed density?

\end{open question}

\begin{open question}[Question 1.8 of B\'etermin-Petrache, \cite{Bet2019AMP}]\label{Open2}
 Is there any non-completely monotone $f$ for which the minimizer of
$E_f[L]$ is the triangular lattice for all $\lambda>0$, among periodic configurations $C$ of
unit density?

\end{open question}

\begin{open question}[Open Problem 1.9 of B\'etermin-Petrache \cite{Bet2019AMP}]\label{OpenA}
 [Stability of crystallization phenomena, with respect to perturbations of $f$] Study and classify natural distances between (or other measures of the size
of perturbations of) interaction kernels $f$ , with respect to which small perturbations
of $f$ can be ensured to preserve the crystallization properties of the kernels, such as
the existence and shape of the global minimum amongst periodic configurations.

\end{open question}

In this paper, we  study  the existence and nonexistence of the minimizer  for the potentials in the
form of difference  of two completely monotone functions. As a consequence, we give affirmative answers to Conjectures \ref{BetConjecture} and \ref{Conjecture2} and answer partially some other open questions listed as above.

Let $\theta (\alpha;z)$ be the theta function defined at (\ref{thetas}). The following is the main result of this paper.

\begin{theorem}\label{Th1} Let $\alpha \geq 1$ and $\beta\in \R$. Consider the minimizing  of difference  of two theta functions with different frequencies, i.e.,
\begin{equation}\label{EFL33} \aligned
\min_{   \mathbb{H} }\Big(\theta (\alpha; z)-\beta\theta (2\alpha; z)\Big).
\endaligned\end{equation}

 There is a critical value $\beta_c=\sqrt2$ (independent of $\alpha$) such that
 \begin{itemize}
   \item if $\beta\leq\beta_c$, the minimizer of the lattice energy functional  is $\frac{1}{2}+i\frac{\sqrt3}{2}$ (up to translation and rotation), which corresponds to the hexagonal lattice;
   \item if $\beta>\beta_c$, the minimizer of the lattice energy functional does not exist.
 \end{itemize}
\end{theorem}

For many physical applications, we state an equivalent form of Theorem \ref{Th1} in the following Theorem \ref{Th2}. 

\begin{theorem}\label{Th2}  Let $\gamma\in(0,1]$ and $\beta \in \R$. Consider the minimizing problem
\begin{equation}\aligned\nonumber
\min_{   \mathbb{H} }\Big(\theta (\gamma; z)-\beta\theta (\frac{1}{2}\gamma; z)\Big).
\endaligned\end{equation}
There is a critical value $\beta_{s}=\frac{\sqrt2}{2}$ (independent of $\alpha$) such that if $\beta\leq\beta_s$, the minimizer of the lattice energy functional  is $\frac{1}{2}+i\frac{\sqrt3}{2}$ (up to translation and rotation), and if  $\beta>\beta_s$, the minimizer of the lattice energy functional does not exist.
\end{theorem}

In the following, we discuss applications to two special potentials: the exponential potential and Yukawa potential. For exponential potentials we have
\begin{theorem}\label{Th3} Let $E_f[L]$ be defined as
$$
E_f(L):=\sum_{\mathbb{P}\in L\backslash\{0\}} f(|\mathbb{P}|^2)
$$
with the area of two dimensional lattice $L$ is normalized to 1.
Consider the potential $f$ has the following form
\begin{equation}\aligned\nonumber
f_\alpha(r):&=e^{-\pi\alpha\cdot r}-\beta e^{-2\pi\alpha\cdot r}, \;\; \alpha\geq1,\\
g_\gamma(r):&=e^{-\pi\gamma\cdot r}-\beta e^{-\frac{1}{2}\pi\gamma\cdot r},\;\;\gamma\in(0,1).
\endaligned\end{equation}
\begin{itemize}
  \item There exists $\beta_c=\sqrt2$ (independent of $\alpha$) such that
\begin{itemize}
  \item if $\beta\leq\beta_c$, the minimizer of $E_{f_\alpha}[L]$ exists and is always hexagonal lattice;
  \item if $\beta>\beta_c$, the minimizer of $E_{f_\alpha}[L]$ does not exists.
\end{itemize}
  \item There exists $\beta_s=\frac{\sqrt2}{2}$ (independent of $\gamma$) such that
\begin{itemize}
  \item if $\beta\leq\beta_s$, the minimizer of $E_{g_\gamma}[L]$ exists and is always hexagonal lattice;
  \item if $\beta>\beta_s$, the minimizer of $E_{g_\gamma}[L]$ does not exists.
\end{itemize}
\end{itemize}

\end{theorem}

\begin{remark}[{\bf An negative answer to Conjecture \ref{BetConjecture}}]

If one regards
\begin{equation}\aligned\nonumber
h_{\alpha_1,\beta}(r):&=e^{-\pi\alpha_1\cdot r}-\beta e^{-\frac{1}{2}\pi\alpha_1\cdot r}, \beta>0,\\
L_{\alpha_2}:&=\alpha_2\cdot L, \alpha_2>0, \alpha:=\alpha_1\cdot\alpha_2<1,
\endaligned\end{equation}
then $h_{\alpha_1,\beta}(r)$ is the difference of two completely monotone functions and is a one well potential
 and
$$
E_{h_{\alpha_1,\beta}}[L_{\alpha_2}]=E_{g_\alpha}[L].
$$
However $E_{h_{\alpha_1,\beta}}[L_\alpha]$ admits minimizer always at hexagonal lattice for any $\beta\in(0,\beta_s]$ as the lattice density $\alpha_2$ changes in $(0,\frac{1}{\alpha_1})$$($one can choose $\alpha_1$ to close $0$ then $\frac{1}{\alpha_1}\rightarrow\infty$$)$. This disproves the hexagonal-rhombic-square-rectangular lattice phase transitions (see e.g.,\cite{Betermin2021JPA,Bet2020}) and gives an negative answer to Conjecture \ref{BetConjecture} and open question 1.16 of B\'etermin-Petrache \cite{Bet2019AMP}. More general potentials of difference of two completely monotone type are shown in Theorem \ref{Th4} via the Laplace transform.

\end{remark}

\begin{remark}[{\bf A negative answer to Conjecture \ref{Conjecture2} on dimension two}] As shown in the Theorem \ref{Th3}, there is no square lattice being minimizer for potential of such form.

\end{remark}
\begin{remark}[{\bf A partial answer to Open Question \ref{Open1}}] Note that
\begin{equation}\aligned\nonumber
h_\beta(r^2)=e^{-\pi\cdot r^2}-\beta e^{-\frac{1}{2}\pi\cdot r^2}, \;\; \beta>0,
\endaligned\end{equation}
is the difference  of two Gaussians $($hence  not a positive superposition of Gaussians$)$, the hexagonal lattice is always the minimizer for any fixed density $\alpha\geq1$. This partially answers open question \ref{Open1}.

\end{remark}

\begin{remark}[{\bf Partial answer on Open Question \ref{OpenA}}] We discuss two aspects: the stability and instability.
\begin{itemize}
  \item $(${\bf Instability under small perturbation: critical parameter}$)$. Let
  \begin{equation}\aligned\nonumber
f_{0,\alpha}(r):&=e^{-\pi\alpha\cdot r}-\sqrt2 e^{-2\pi\alpha\cdot r}, \;\; \alpha\geq1,\\
g_{0,\gamma}(r):&=e^{-\pi\gamma\cdot r}-\frac{\sqrt2}{2} e^{-\frac{1}{2}\pi\gamma\cdot r}\;\; \gamma\in(0,1).
\endaligned\end{equation}
A small perturbation of  $f_{0,\alpha}(r)$ and $g_{0,\gamma}(r)$ from left hand side by $\varepsilon e^{-2 c\cdot r}$ will lead to the minimizer of $E_f[L]$ does not exist. Namely, let
  \begin{equation}\aligned\nonumber
f_{0,\alpha,\varepsilon}(r):&=e^{-\pi\alpha\cdot r}-\sqrt2 e^{-2\pi\alpha\cdot r}-\varepsilon e^{-2 c\cdot r}, \;\; \alpha\geq1\\
g_{0,\gamma,\varepsilon}(r):&=e^{-\pi\gamma\cdot r}-\frac{\sqrt2}{2} e^{-\frac{1}{2}\pi\gamma\cdot r}-\varepsilon e^{-2 c\cdot r}\;\; \gamma\in(0,1),
\endaligned\end{equation}
for $\forall c>0, \forall\varepsilon>0$ be a small perturbation of $f_{0,\alpha}(r), g_{0,\gamma}(r)$, the minimizers of $E_{f_{0,\alpha,\varepsilon}}[L]$ and $E_{g_{0,\gamma,\varepsilon}}[\Lambda]$ do not exist and the minimizers of $E_{f_{0,\alpha}}[L]$ and $E_{g_{0,\gamma}}[L]$ are both hexagonal lattice. In this sense, the minimizers of $E_{f_{0,\alpha}}[L]$ are instable under small perturbation as above.

  \item $(${\bf Stability under small perturbation: subcritical parameter}$)$.
Assume $\beta<\frac{\sqrt2}{2}$. Let $|\varepsilon|\leq\frac{\sqrt2}{2}-\beta$ and
  \begin{equation}\aligned\nonumber
g_{\alpha,\varepsilon}(r):&=e^{-\pi\alpha\cdot r}-\beta e^{-2\pi\alpha\cdot r}\pm\varepsilon e^{-2\pi\alpha\cdot r}, \;\; \alpha\geq1,
\endaligned\end{equation}
be a small perturbation of $g_{\alpha}(r)$.
Then the minimizer of $E_{g_{\alpha,\varepsilon}(r)}[L]$ is still the hexagonal lattice, i.e., the minimizer of $E_{g_{\alpha}(r)}[L]$ is stable under small perturbation as above.

\end{itemize}
\end{remark}

\begin{remark} The numerical study of the potential
 \begin{equation}\aligned\nonumber
f_\alpha(r):&=e^{-\pi\alpha\cdot r}-\beta e^{-2\pi\alpha\cdot r}
\endaligned\end{equation}
is performed as an important cases in B\'etermin-Faulhuber-Kn$\ddot{u}$pfer \cite{Bet2020}, see Figures 3, 8 and 10 of \cite{Bet2020}.

\end{remark}

Using the free parameter $\alpha$ of Theorem \ref{Th3}, we proceed Theorem \ref{Th2} and \ref{Th3} to a general form by the Laplace transform (inspired by B\'etermin \cite{Bet2016}). There is a difference between Theorem \ref{Th4} and Theorems \ref{Th1}, \ref{Th2}, \ref{Th3}. In the former, one does not know the parameter is optimal or not, and in the latter, the parameters are optimal as stated in the Theorems.

\begin{theorem}\label{Th4}
Let $E_f[L]$ be defined as
$$
E_f(L):=\sum_{\mathbb{P}\in L \backslash\{0\}} f(|\mathbb{P}|^2)
$$
with the area of two dimensional lattice $L$ is normalized to 1.
Consider the potential $f_{\alpha,P}, g_{\alpha,P}$ have the following form
\begin{equation}\aligned\label{FFGG}
f_{\alpha,P}(r):&=\int_1^\infty \Big(\big(e^{-\pi\alpha x\cdot r}-\beta e^{-2\pi\alpha x\cdot r}\big)\cdot P(x)\Big)dx, \;\; \alpha\geq1,\\
g_{\gamma,P}(r):&=\int_0^1 \Big(\big(e^{-\pi\gamma x\cdot r}-\beta e^{-\frac{1}{2}\pi\gamma x\cdot r}\big)\cdot P(x)\Big)dx, \;\; \gamma\in(0,1),
\endaligned\end{equation}
where the $P(x)$ is any real function(not necessarily continuous) such that $f_{\alpha,P}(r), g_{\alpha,P}(r)$ are finite and
$$P(x)\geq0.$$

Then there exists $\beta_c=\sqrt2, \beta_s=\frac{\sqrt2}{2}$ and nonnegative function $P$ such that
\begin{itemize}
  \item if $\beta\leq\beta_c$, the minimizer of $E_{f_{\alpha,P}}[\Lambda]$ exists and is always hexagonal lattice.
  \item if $\beta\leq\beta_s$, the minimizer of $E_{g_{\gamma,P}}[\Lambda]$ exists and is always hexagonal lattice.
\end{itemize}
\end{theorem}

\begin{remark}[{\bf Partial answer on open question 1.6 of B\'etermin-Petrache \cite{Bet2019AMP}}] Theorem \ref{Th4} partially answers open question 1.6 of B\'etermin-Petrache \cite{Bet2019AMP} on minimizers of difference of two Laplace transform of the potentials.

\end{remark}

\begin{remark}[Connection to G-type potentials] The potential introduced by \eqref{FFGG} are a subclass of G-type potentials $($see Chapter 10 of monograph \cite{B10}$)$. Here we show that under these potentials the minimizer of the lattice energy are hexagonal lattice under suitable competing strength $\beta$.

\end{remark}

\begin{remark}[\bf Partial answer on Open Question \ref{Open2}] Since here $f_{\alpha,P}(r),(\alpha\geq1)$ is the difference of two completely monotone functions, and hence not completely monotone. The minimizer of $E_{f_{\alpha,P}(r)}[\lambda\cdot L]$ is hexagonal lattice for all $\lambda\geq1$. This gives partial answer on Open question \ref{Open2}.

\end{remark}

\begin{remark}[{\bf More potentials to answer Conjecture \ref{BetConjecture}}] Theorem \ref{Th4} provides general potentials to a negative answer to Conjecture \ref{BetConjecture}.
\end{remark}

A particular  application of  Theorem \ref{Th4} is the classical  Yukawa potential case.
\begin{corollary}[{\bf Yukawa potential$\{\cong P(x)\equiv1\}$ of Theorem \ref{Th4}}]\label{Coro1}

Let $E_f[L]$ be defined as
$$
E_f(L):=\sum_{\mathbb{P}\in L\backslash\{0\}} f(|\mathbb{P}|^2)
$$
with the area of two dimensional lattice $L$ is normalized to 1.
Consider the potential $f_{1,\alpha}, g_{1,\alpha}$ have the following form
\begin{equation}\aligned\nonumber
f_{1,\alpha}(r):&=\frac{e^{-\pi\alpha r}}{r}-\beta\frac{e^{-2\pi\alpha r}}{2r}, \;\; \alpha\geq1.
\endaligned\end{equation}
Then there exists $\beta_c=\sqrt2$ independent of parameter $\alpha$ such that
\begin{itemize}
  \item if $\beta\leq\beta_c$, the minimizer of $E_{f_{1,\alpha}}[L]$ exists and is always hexagonal lattice.
\end{itemize}
\end{corollary}

\begin{remark} The rigorous results on Yukawa potential of the minimizer of the crystal energy $E_f[L]$, as far as we know, is initiated by B\'etermin \cite{Bet2016}. Here Corollary \ref{Coro1} improves the result in \cite{Bet2016} on some aspects. Note that we provide an effective way to prove the crystallization of hexagonal lattice under Yukawa potential.
\end{remark}

\begin{remark} In using the results of Corollary \ref{Coro1} and combining the method of B\'etermin \cite{Bet2016}, one can obtain more general results on minimization results under Yukawa potential.
\end{remark}

Theorem \ref{Th1} can be extended as follows by iteration.
\begin{theorem}\label{ThA} Consider the minimizing problem of difference of two theta functions
\begin{equation}\aligned\nonumber
\min_{   \mathbb{H} }\Big(\theta (\alpha; z)-\beta\theta (2^k\alpha; z)\Big),\;\;\hbox{for any}\;\alpha\geq1,k\geq0, k\in\mathbb{Z},\;\beta\in(-\infty,\infty).
\endaligned\end{equation}
There is a critical value $\beta_{A}=\sqrt{2^k}$ independent of $\alpha$ such that
 \begin{itemize}
   \item if $\beta\leq\beta_A$, the minimizer of the lattice energy functional  is $\frac{1}{2}+i\frac{\sqrt3}{2}$ up to translation and rotation, this minimizer corresponds to $\Lambda$ is the hexagonal lattice;
   \item if $\beta>\beta_A$, the minimizer of the lattice energy functional does not exist.
 \end{itemize}
\end{theorem}

The paper is organized as follows: in Section 2, we state some basic preliminaries about the functional $ \theta (\alpha;z)-\beta \theta (2\alpha;z)$.
In Section 3, we prove that the minimization problem on the fundamental region (see \eqref{Fd3} and figure \ref{f-FFF}) can be reduced to a vertical line (see figure \ref{f-FFF}). (See Theorem \ref{2Th2}.)
In Section 4, we prove that the minimization problem on the vertical line can be reduced to a particular point (hexagonal point (see figure \ref{f-FFF})). We develop effective methods and delicate analysis to obtain the estimates, which can be generalized to solve related problems. As a consequence we prove Theorem \ref{Th1}. Finally Section 5 contains proofs of remaining Theorems. 

\begin{figure}
\centering
 \includegraphics[scale=0.58]{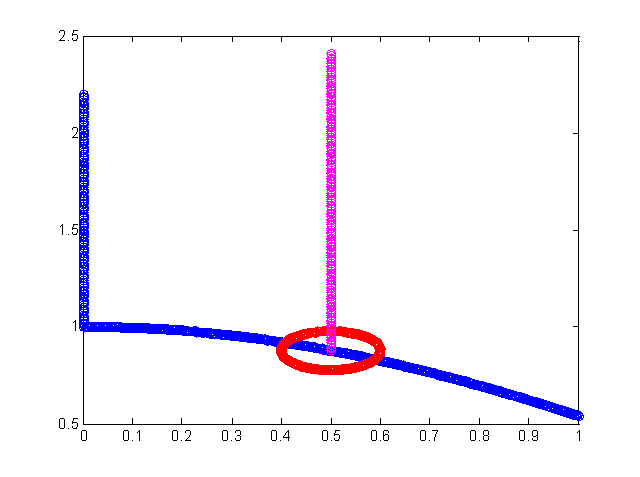}\includegraphics[scale=0.58]{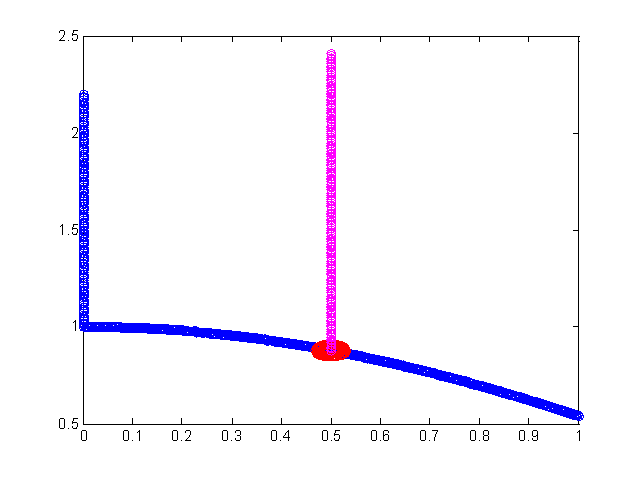}
 \caption{Location of the fundamental region and hexagonal point.}
\label{f-FFF}
\end{figure}

\section{Preliminaries }
\setcounter{equation}{0}

In this section we collect  some simple symmetries of the theta function $\theta (s; z)$ and the associated fundamental domain, and also the properties of Jacobi theta functions to be used in later sections.

Let
$
\mathbb{H}
$
 denote the upper half plane and  $\mathcal{S} $ denote the modular group
\begin{equation}\aligned\label{modular}
\mathcal{S}:=SL_2(\mathbb{Z})=\{
\left(
  \begin{array}{cc}
    a & b \\
    c & d \\
  \end{array}
\right), ad-bc=1, a, b, c, d\in\mathbb{Z}
\}.
\endaligned\end{equation}

We use the following definition of fundamental domain which is slightly different from the classical definition (see \cite{Mon1988}):
\begin{definition} [page 108, \cite{Eva1973}]
The fundamental domain associated to group $G$ is a connected domain $\mathcal{D}$ satisfies
\begin{itemize}
  \item For any $z\in\mathbb{H}$, there exists an element $\pi\in G$ such that $\pi(z)\in\overline{\mathcal{D}}$;
  \item Suppose $z_1,z_2\in\mathcal{D}$ and $\pi(z_1)=z_2$ for some $\pi\in G$, then $z_1=z_2$ and $\pi=\pm Id$.
\end{itemize}
\end{definition}

By Definition 1, the fundamental domain associated to modular group $\mathcal{S}$ is
\begin{equation}\aligned\label{Fd1}
\mathcal{D}_{\mathcal{S}}:=\{
z\in\mathbb{H}: |z|>1,\; -\frac{1}{2}<x<\frac{1}{2}
\}
\endaligned\end{equation}
which is open.  Note that the fundamental domain can be open. (See [page 30, \cite{Apo1976}].)

Next we introduce another group related  to the functionals $\theta(\alpha;z)$. The generators of the group are given by
\begin{equation}\aligned\label{GroupG1}
\mathcal{G}: \hbox{the group generated by} \;\;\tau\mapsto -\frac{1}{\tau},\;\; \tau\mapsto \tau+1,\;\;\tau\mapsto -\overline{\tau}.
\endaligned\end{equation}

It is easy to see that
the fundamental domain associated to group $\mathcal{G}$ denoted by $\mathcal{D}_{\mathcal{G}}$ is
\begin{equation}\aligned\label{Fd3}
\mathcal{D}_{\mathcal{G}}:=\{
z\in\mathbb{H}: |z|>1,\; 0<x<\frac{1}{2}
\}.
\endaligned\end{equation}

The following lemma characterizes the fundamental symmetries of the theta functions $\theta (s; z)$. The proof is easy so we omit it.
\begin{lemma}\label{G111} For any $s>0$, any $\gamma\in \mathcal{G}$ and $z\in\mathbb{H}$,
$\ \theta (s; \gamma(z))=\theta (s;z)$.
\end{lemma}

Let
\begin{equation}\label{Wbeta} \aligned
\mathrm{W}_\beta(\alpha; z):=\theta(\alpha;z)-\beta\theta(2\alpha;z).
\endaligned\end{equation}

\begin{lemma}\label{Geee}  For any $\alpha>0$, any $\gamma\in \mathcal{G}$ and $z\in\mathbb{H}$,
$\mathrm{W}_\beta(\alpha; \gamma(z))=\mathrm{W}_\beta (\alpha;z)$.
\end{lemma}

We also need some delicate analysis of the Jacobi theta function.

We first recall the following well-known Jacobi triple product formula:
\begin{equation}\aligned\label{Jacob1}
\prod_{m=1}^\infty(1-x^{2m})(1+x^{2m-1}y^2)(1+\frac{x^{2m-1}}{y^2})=\sum_{n=-\infty}^\infty x^{n^2} y^{2n}
 \endaligned\end{equation}
for complex numbers $x,y$ with $|x|<1$, $y\neq0$.

The Jacobi theta function is defined as
\begin{equation}\aligned\nonumber
\vartheta_J(z;\tau):=\sum_{n=-\infty}^\infty e^{i\pi n^2 \tau+2\pi i n z},
 \endaligned\end{equation}
and the classical one-dimensional theta function  is given by
\begin{equation}\aligned\label{TXY}
\vartheta(X;Y):=\vartheta_J(Y;iX)=\sum_{n=-\infty}^\infty e^{-\pi n^2 X} e^{2n\pi i Y}.
 \endaligned\end{equation}
Hence by the Jacobi triple product formula \eqref{Jacob1}, it holds
\begin{equation}\aligned\label{Product}
\vartheta(X;Y)=\prod_{n=1}^\infty(1-e^{-2\pi n X})(1+e^{-2(2n-1)\pi X}+2e^{-(2n-1)\pi X}\cos(2\pi Y)).
 \endaligned\end{equation}

The following Lemmas \ref{LemmaTTT} and \ref{LemmaT2} are proved in  \cite{Luo2022}.
\begin{lemma}\label{LemmaTTT} Assume $X>\frac{1}{5}$. If $\sin(2\pi Y)>0$, then
\begin{equation}\aligned\nonumber
-\overline\vartheta(X)\sin(2\pi Y)\leq\frac{\partial}{\partial Y}\vartheta(X;Y)\leq-\underline\vartheta(X)\sin(2\pi Y).
 \endaligned\end{equation}
If $\sin(2\pi Y)<0$, then
\begin{equation}\aligned\nonumber
-\underline\vartheta(X)\sin(2\pi Y)\leq\frac{\partial}{\partial Y}\vartheta(X;Y)\leq-\overline\vartheta(X)\sin(2\pi Y).
 \endaligned\end{equation}
Here
\begin{equation}\aligned\nonumber
\underline\vartheta(X):=4\pi e^{-\pi X}(1-\mu(X)), \;\; \overline\vartheta(X):=4\pi e^{-\pi X}(1+\mu(X)),
 \endaligned\end{equation}
and
\begin{equation}\label{mmmx}
\mu(X):=\sum_{n=2}^\infty n^2 e^{-\pi(n^2-1)X}.
\end{equation}

\end{lemma}


\begin{lemma}\label{LemmaT2}
Assume $X<\min\{\frac{\pi}{\pi+2},\frac{\pi}{4\log\pi}\}=\frac{\pi}{\pi+2}$. If $\sin(2\pi Y)>0$, then
\begin{equation}\aligned\nonumber
-\overline\vartheta(X)\sin(2\pi Y)\leq\frac{\partial}{\partial Y}\vartheta(X;Y)\leq-\underline\vartheta(X)\sin(2\pi Y).
 \endaligned\end{equation}
If $\sin(2\pi Y)<0$, then
\begin{equation}\aligned\nonumber
-\underline\vartheta(X)\sin(2\pi Y)\leq\frac{\partial}{\partial Y}\vartheta(X;Y)\leq-\overline\vartheta(X)\sin(2\pi Y).
 \endaligned\end{equation}
Here
\begin{equation}\aligned\nonumber
\underline\vartheta(X):=\pi e^{-\frac{\pi}{4X}}X^{-\frac{3}{2}};\;\; \overline\vartheta(X):=X^{-\frac{3}{2}}.
 \endaligned\end{equation}

\end{lemma}


\section{The transversal monotonicity}
\setcounter{equation}{0}
Let $\mathcal{D}_{\mathcal{G}}:=\{
z\in\mathbb{H}: |z|>1,\; 0<x<\frac{1}{2}
\}$ be the fundamental domain associated to the group $\mathcal{G}$. Define the vertical line
\begin{equation}\aligned\label{Gfff}
\Gamma:=\{
z\in\mathbb{H}: \Re(z)=\frac{1}{2},\; \Im(z)\geq\frac{\sqrt3}{2}
\}.
\endaligned\end{equation}

By the group invariance (Lemma \ref{Geee}), one has
\begin{equation}\aligned\label{Domain1}
\min_{z\in\mathbb{H}}\Big(\theta(\alpha;z)-\beta\theta(2\alpha;z)\Big)
=\min_{z\in\overline{\mathcal{D}_{\mathcal{G}}}}\Big(\theta(\alpha;z)-\beta\theta(2\alpha;z)\Big).
\endaligned\end{equation}

Let $\mu (X)$ be defined in (\ref{mmmx}) and
\begin{equation} \label{beta0} \aligned
\beta_0:=\min\begin{cases}
\frac{\pi e^{2\pi}-4e^{-6\pi}}{2}\\
\frac{\sqrt2 e^{\frac{\sqrt3\pi}{4}}(1-\mu(\frac{1}{2}))-4e^{-\frac{5\sqrt3\pi}{4}}(1+\mu(\frac{1}{4}))}{1+\mu(\frac{1}{4})}\\
\frac{4\sqrt2\pi e^\pi\big(1-\mu(\frac{1}{4})\big)-16\sqrt2\pi e^{-\frac{7\pi}{2}}\big(1+\mu(\frac{1}{4})\big)}{64}.
\end{cases}
\endaligned\end{equation}
Numerically,
\begin{equation}\aligned\nonumber
\beta_0:\cong3.801819\cdots.
\endaligned\end{equation}
As we shall see in the next Section, one only needs to require that $\beta_0\geq\sqrt2$
to obtain all the main results.

In this section, we aim to establish that
\begin{theorem}\label{2Th1} Assume that $\alpha\geq1$. Let $\beta_0$ be defined at (\ref{beta0}). Then   for $\beta<\beta_0$,
\begin{equation}\aligned\nonumber
\min_{z\in\mathbb{H}}\Big(\theta(\alpha;z)-\beta\theta(2\alpha;z)\Big)=\min_{z\in\overline{\mathcal{D}_{\mathcal{G}}}}\Big(\theta(\alpha;z)-\beta\theta(2\alpha;z)\Big)=\min_{z\in\Gamma}\Big(\theta(\alpha;z)-\beta\theta(2\alpha;z)\Big),
\endaligned\end{equation}
where $\Gamma$ is a vertical line and defined at \eqref{Gfff}.
\end{theorem}

The proof of Theorem \ref{2Th1} follows from the following monotonicity result
\begin{theorem}\label{2Th3}
Assume that $\alpha\geq1$. Then for $\beta<\beta_0$
\begin{equation}\aligned\nonumber
\frac{\partial}{\partial x}\Big(\theta(\alpha;z)-\beta\theta(2\alpha;z)\Big)
<0, \;\;\hbox{for}\;\; z\in
{\mathcal{D}_{\mathcal{G}}}.
\endaligned\end{equation}
\end{theorem}

In the rest of this Section, we prove Theorem \ref{2Th3}.

\subsection{The estimates}

We provide an exponential expansion of the theta function, which is useful in our estimates.
\begin{lemma}\label{Lemma31}  We have the following  exponential expansion of theta function:
\begin{equation}\aligned\label{PPP3}
\theta (\alpha;z)=2\sqrt{\frac{y}{\alpha}}\sum_{n=1}^\infty e^{-\alpha \pi y n^2}\vartheta(\frac{y}{\alpha};nx)+\sqrt{\frac{y}{\alpha}}\vartheta(\frac{y}{\alpha};0).
\endaligned\end{equation}
\end{lemma}

\begin{proof}
In view of \eqref{TXY}, by Poisson Summation Formula, one has
\begin{equation}\aligned\label{PPP2}
\vartheta(X;Y)=X^{-\frac{1}{2}}\sum_{n\in \mathbb{Z}} e^{-\frac{\pi(n-Y)^2}{X}}.
\endaligned\end{equation}
Then
\begin{equation}\aligned\nonumber
\theta (\alpha;z)&=\sum_{(m,n)\in\mathbb{Z}^2} e^{-\alpha \pi\frac{1}{y }|nz+m|^2}
=\sum_{n\in\mathbb{Z}}e^{-\alpha \pi y n^2}\sum_{m\in\mathbb{Z}} e^{-\frac{\alpha \pi (nx+m)^2}{y}}\\
&=\sqrt{\frac{y}{\alpha}}\sum_{n\in\mathbb{Z}}e^{-\alpha \pi y n^2}\vartheta(\frac{y}{\alpha};-nx)=\sqrt{\frac{y}{\alpha}}\sum_{n\in\mathbb{Z}}e^{-\alpha \pi y n^2}\vartheta(\frac{y}{\alpha};nx)\\
&=2\sqrt{\frac{y}{\alpha}}\sum_{n=1}^\infty e^{-\alpha \pi y n^2}\vartheta(\frac{y}{\alpha};nx)+\sqrt{\frac{y}{\alpha}}\vartheta(\frac{y}{\alpha};0).
\endaligned\end{equation}

\end{proof}

\begin{lemma}\label{Lemma32} We have the following  identity for derivative of the theta function with respect to $x$
\begin{equation}\nonumber\aligned
-\frac{\partial}{\partial x}\big(
\theta(\alpha;z)-\beta\theta(2\alpha;z)
\big)
=2\sqrt{\frac{y}{2\alpha}}e^{-\pi\alpha y}\mathcal{E}_{\alpha,\beta,x}(z).
\endaligned\end{equation}
Here
\begin{equation}\nonumber\aligned
\mathcal{E}_{\alpha,\beta,x}(z):
=
\sqrt2\sum_{n=1}^\infty ne^{-\alpha\pi y(n^2-1)}\big(-\frac{\partial}{\partial Y}\vartheta(\frac{y}{\alpha};nx)\big)-\beta\sum_{n=1}^\infty ne^{-\alpha\pi y(2n^2-1)}\big(-\frac{\partial}{\partial Y}\vartheta(\frac{y}{2\alpha};nx)\big).
\endaligned\end{equation}
\end{lemma}

\begin{proof} By Lemma \ref{Lemma31},
\begin{equation}\nonumber\aligned
-\frac{\partial}{\partial x}\big(
\theta(\alpha;z)-\beta\theta(2\alpha;z)
\big)
&=-\frac{\partial}{\partial x}
\Big(
2\sqrt{\frac{y}{\alpha}}\sum_{n=1}^\infty e^{-\alpha \pi y n^2}\vartheta(\frac{y}{\alpha};nx)
-2\beta\sqrt{\frac{y}{2\alpha}}\sum_{n=1}^\infty e^{-2\alpha \pi y n^2}\vartheta(\frac{y}{2\alpha};nx)
\Big)\\
&=2\sqrt{\frac{y}{2\alpha}}e^{-\pi\alpha y}(-\frac{\partial}{\partial x})
\Big(
\sqrt2\sum_{n=1}^\infty e^{-\alpha\pi y(n^2-1)}\vartheta(\frac{y}{\alpha};nx)\\
&\;\;\;\;-\beta\sum_{n=1}^\infty e^{-\alpha\pi y(2n^2-1)}\vartheta(\frac{y}{2\alpha};nx)
\Big)\\
&=2\sqrt{\frac{y}{2\alpha}}e^{-\pi\alpha y}
\Big(
\sqrt2\sum_{n=1}^\infty ne^{-\alpha\pi y(n^2-1)}\big(-\frac{\partial}{\partial Y}\vartheta(\frac{y}{\alpha};nx)\big)\\
&\;\;\;\;-\beta\sum_{n=1}^\infty ne^{-\alpha\pi y(2n^2-1)}\big(-\frac{\partial}{\partial Y}\vartheta(\frac{y}{2\alpha};nx)\big)
\Big)\\
&=2\sqrt{\frac{y}{2\alpha}}e^{-\pi\alpha y}\mathcal{E}_{\alpha,\beta,x}(z).
\endaligned\end{equation}

\end{proof}

\begin{lemma}\label{Lemma33} For $x\in[0,\frac{1}{3}]$, $z\in\overline{\mathcal{D}_{\mathcal{G}}}$,
 \begin{equation}\nonumber\aligned
\mathcal{E}_{\alpha,\beta,x}(z)\geq
\sin(2\pi x)
 \Big(
 \sqrt2\underline\vartheta(\frac{y}{\alpha})
 -(\beta+\sigma_2) e^{-\alpha\pi y}\overline\vartheta(\frac{y}{2\alpha})-(4\sqrt2e^{-3\alpha\pi y}+\sigma_1)\overline\vartheta(\frac{y}{\alpha})
 \Big),
\endaligned\end{equation}
where
 \begin{equation}\label{sigma12} \aligned
\sigma_1:=\sqrt2\sum_{n=2}^\infty n^2e^{-\alpha\pi y(n^2-1)},
\sigma_2:=\beta\sum_{n=2}^\infty n^2e^{-\alpha\pi y(2n^2-1)}.
\endaligned\end{equation}

\end{lemma}
\begin{proof}

We decompose

 \begin{equation}\nonumber\aligned
\mathcal{E}_{\alpha,\beta,x}(z):
=&\sqrt2(-\frac{\partial}{\partial Y})\vartheta(\frac{y}{\alpha};x)-be^{-\pi\alpha y}(-\frac{\partial}{\partial Y})\vartheta(\frac{y}{2\alpha};x)
+2\sqrt2e^{-3\pi\alpha y}(-\frac{\partial}{\partial Y})\vartheta(\frac{y}{\alpha};x)\\
&+\sqrt2\sum_{n=2}^\infty ne^{-\alpha\pi y(n^2-1)}\big(-\frac{\partial}{\partial Y}\vartheta(\frac{y}{\alpha};nx)\big)
-\beta\sum_{n=2}^\infty ne^{-\alpha\pi y(2n^2-1)}\big(-\frac{\partial}{\partial Y}\vartheta(\frac{y}{2\alpha};nx)\big)\\
=&\mathcal{E}^{a,1}_{\alpha,\beta,x}(z)+\mathcal{E}^{b,1}_{\alpha,\beta,x}(z).
\endaligned\end{equation}
where $\mathcal{E}^{a,1}_{\alpha,\beta,x}(z)$ and $\mathcal{E}^{b,1}_{\alpha,\beta,x}(z)$ are defined at the last equality.

For $\mathcal{E}^{b,1}_{\alpha,\beta,x}(z)$ we estimate as follows
 \begin{equation}\nonumber\aligned
|\mathcal{E}^{b,1}_{\alpha,\beta,x}(z)|
&\leq\sqrt2\sum_{n=2}^\infty ne^{-\alpha\pi y(n^2-1)}\overline\vartheta(\frac{y}{\alpha})|\sin(2\pi nx)|
+\beta\sum_{n=2}^\infty ne^{-\alpha\pi y(2n^2-1)}\overline\vartheta(\frac{y}{2\alpha})|\sin(2\pi nx)|\\
&\leq\sqrt2\sum_{n=2}^\infty n^2e^{-\alpha\pi y(n^2-1)}\overline\vartheta(\frac{y}{\alpha})\sin(2\pi x)
+\beta\sum_{n=2}^\infty n^2e^{-\alpha\pi y(2n^2-1)}\overline\vartheta(\frac{y}{2\alpha})\sin(2\pi x)\\
&\leq\sigma_1\cdot\overline\vartheta(\frac{y}{\alpha})\sin(2\pi x)
+\sigma_2\cdot\overline\vartheta(\frac{y}{2\alpha})\sin(2\pi x),
\endaligned\end{equation}
where $\sigma_1$ and $\sigma_2$ are defined at (\ref{sigma12}).

For $\mathcal{E}^{b,1}_{\alpha,\beta,x}(z)$  we have
 \begin{equation}\nonumber\aligned
\mathcal{E}^{a,1}_{\alpha,\beta,x}(z)
&\geq\sqrt2\underline\vartheta(\frac{y}{\alpha})\sin(2\pi x)-\beta e^{-\pi\alpha y}\overline\vartheta(\frac{y}{2\alpha})\sin(2\pi x)
+2\sqrt2e^{-3\pi\alpha y}\underline\vartheta(\frac{y}{\alpha})\sin(4\pi x)\\
&\geq\sin(2\pi x)\Big(\sqrt2\underline\vartheta(\frac{y}{\alpha})-\beta e^{-\pi\alpha y}\overline\vartheta(\frac{y}{2\alpha})\Big).
\endaligned\end{equation}

Therefore,
 \begin{equation}\nonumber\aligned
\mathcal{E}_{\alpha,\beta,x}(z)
=&\mathcal{E}^{a,1}_{\alpha,\beta,x}(z)+\mathcal{E}^{b,1}_{\alpha,\beta,x}(z)\\
\geq&\sin(2\pi x)
\Big(\sqrt2\underline\vartheta(\frac{y}{\alpha})-(\beta+\sigma_2)e^{-\pi\alpha y}\overline\vartheta(\frac{y}{2\alpha})-\sigma_1\overline\vartheta(\frac{y}{\alpha})\Big).
\endaligned\end{equation}
\end{proof}

\begin{lemma}\label{Lemma34} For $x\in[\frac{1}{3}, \frac{1}{2}]$, $z\in\overline{\mathcal{D}_{\mathcal{G}}}$,

 \begin{equation}\nonumber\aligned
\mathcal{E}_{\alpha,\beta,x}(z)\geq
\sin(2\pi x)
 \Big(
 \sqrt2\underline\vartheta(\frac{y}{\alpha})
 -(\beta+\sigma_4) e^{-\alpha\pi y}\overline\vartheta(\frac{y}{2\alpha})-(4\sqrt2e^{-3\alpha\pi y}+\sigma_3)\overline\vartheta(\frac{y}{\alpha})
 \Big).
\endaligned\end{equation}
Here
 \begin{equation}\label{sigma34}\aligned
\sigma_3:=
\sqrt2\sum_{n=4}^\infty n^2e^{-\alpha\pi y(n^2-1)},
\sigma_4:=\beta\sum_{n=3}^\infty n^2e^{-\alpha\pi y(2n^2-1)}.
\endaligned\end{equation}
\end{lemma}
\begin{proof}
In this case, $\big(-\frac{\partial}{\partial Y}\vartheta(y;nx)\big)\geq0$ is positive for $n=1,3$, $y>0$ and negative for $n=2$, $y>0$. We then decompose
 \begin{equation}\nonumber\aligned
\mathcal{E}_{\alpha,\beta,x}(z):
=&\sqrt2\sum_{n=1}^\infty ne^{-\alpha\pi y(n^2-1)}\big(-\frac{\partial}{\partial Y}\vartheta(\frac{y}{\alpha};nx)\big)
-\beta\sum_{n=1}^\infty ne^{-\alpha\pi y(2n^2-1)}\big(-\frac{\partial}{\partial Y}\vartheta(\frac{y}{2\alpha};nx)\big)\\
=& \sqrt2e^{-\alpha\pi y}\big(-\frac{\partial}{\partial Y}\vartheta(\frac{y}{\alpha};x)\big)+2\sqrt2e^{-3\alpha\pi y}\big(-\frac{\partial}{\partial Y}\vartheta(\frac{y}{\alpha};2x)\big)+3\sqrt2e^{-8\alpha\pi y}\big(-\frac{\partial}{\partial Y}\vartheta(\frac{y}{\alpha};3x)\big)\\
&
-\beta e^{-\alpha\pi y}\big(-\frac{\partial}{\partial Y}\vartheta(\frac{y}{2\alpha};x)\big)-2\beta e^{-7\alpha\pi y}\big(-\frac{\partial}{\partial Y}\vartheta(\frac{y}{2\alpha};2x)\big)\\
&+\sqrt2\sum_{n=4}^\infty ne^{-\alpha\pi y(n^2-1)}\big(-\frac{\partial}{\partial Y}\vartheta(\frac{y}{\alpha};nx)\big)
-\beta\sum_{n=3}^\infty ne^{-\alpha\pi y(2n^2-1)}\big(-\frac{\partial}{\partial Y}\vartheta(\frac{y}{2\alpha};nx)\big)\\
&:=\mathcal{E}^{a,2}_{\alpha,\beta,x}(z)+\mathcal{E}^{b,2}_{\alpha,\beta,x}(z).
\endaligned\end{equation}
 where $\mathcal{E}^{a,2}_{\alpha,\beta,x}(z)$ and $\mathcal{E}^{b,2}_{\alpha,\beta,x}(z)$ are defined at the last equality.

For $\mathcal{E}^{a,2}_{\alpha,\beta,x}(z)$ we have
 \begin{equation}\nonumber\aligned
\mathcal{E}^{a,2}_{\alpha,\beta,x}(z):
=& \sqrt2\big(-\frac{\partial}{\partial Y}\vartheta(\frac{y}{\alpha};x)\big)+2\sqrt2e^{-3\alpha\pi y}\big(-\frac{\partial}{\partial Y}\vartheta(\frac{y}{\alpha};2x)\big)+3\sqrt2e^{-8\alpha\pi y}\big(-\frac{\partial}{\partial Y}\vartheta(\frac{y}{\alpha};3x)\big)\\
&
-\beta e^{-\alpha\pi y}\big(-\frac{\partial}{\partial Y}\vartheta(\frac{y}{2\alpha};x)\big)-2\beta e^{-7\alpha\pi y}\big(-\frac{\partial}{\partial Y}\vartheta(\frac{y}{2\alpha};2x)\big).
\endaligned\end{equation}

For $\mathcal{E}^{b,2}_{\alpha,\beta,x}(z)$,
 \begin{equation}\nonumber\aligned
\mathcal{E}^{b,2}_{\alpha,\beta,x}(z):
=\sqrt2\sum_{n=4}^\infty ne^{-\alpha\pi y(n^2-1)}\big(-\frac{\partial}{\partial Y}\vartheta(\frac{y}{\alpha};nx)\big)
-\beta\sum_{n=3}^\infty ne^{-\alpha\pi y(2n^2-1)}\big(-\frac{\partial}{\partial Y}\vartheta(\frac{y}{2\alpha};nx)\big).
\endaligned\end{equation}

A lower bound of  $\mathcal{E}^{a,2}_{\alpha,\beta,x}(z)$ yields
 \begin{equation}\nonumber\aligned
\mathcal{E}^{a,2}_{\alpha,\beta,x}(z):
=& \sqrt2\big(-\frac{\partial}{\partial Y}\vartheta(\frac{y}{\alpha};x)\big)+2\sqrt2e^{-3\alpha\pi y}\big(-\frac{\partial}{\partial Y}\vartheta(\frac{y}{\alpha};2x)\big)+3\sqrt2e^{-8\alpha\pi y}\big(-\frac{\partial}{\partial Y}\vartheta(\frac{y}{\alpha};3x)\big)\\
&
-\beta e^{-\alpha\pi y}\big(-\frac{\partial}{\partial Y}\vartheta(\frac{y}{2\alpha};x)\big)-2\beta e^{-7\alpha\pi y}\big(-\frac{\partial}{\partial Y}\vartheta(\frac{y}{2\alpha};2x)\big)\\
\geq&
 \sqrt2\big(-\frac{\partial}{\partial Y}\vartheta(\frac{y}{\alpha};x)\big)+2\sqrt2e^{-3\alpha\pi y}\big(-\frac{\partial}{\partial Y}\vartheta(\frac{y}{\alpha};2x)\big)
 -\beta e^{-\alpha\pi y}\big(-\frac{\partial}{\partial Y}\vartheta(\frac{y}{2\alpha};x)\big)\\
 \geq&
 \sqrt2\underline\vartheta(\frac{y}{\alpha})\sin(2\pi x)+2\sqrt2e^{-3\alpha\pi y}\overline\vartheta(\frac{y}{\alpha})\sin(4\pi x)
 -\beta e^{-\alpha\pi y}\overline\vartheta(\frac{y}{2\alpha})\sin(2\pi x)\\
 \geq&
 \sin(2\pi x)
 \Big(
 \sqrt2\underline\vartheta(\frac{y}{\alpha})
 -\beta e^{-\alpha\pi y}\overline\vartheta(\frac{y}{2\alpha})-4\sqrt2e^{-3\alpha\pi y}\overline\vartheta(\frac{y}{\alpha})
 \Big).
\endaligned\end{equation}

A upper bound of $\mathcal{E}^{b,2}_{\alpha,\beta,x}(z)$ is given by
 \begin{equation}\nonumber\aligned
|\mathcal{E}^{b,2}_{\alpha,\beta,x}(z)|
&\leq\sqrt2\sum_{n=4}^\infty ne^{-\alpha\pi y(n^2-1)}\overline\vartheta(\frac{y}{\alpha})|\sin(2n\pi x)|
+\beta\sum_{n=3}^\infty ne^{-\alpha\pi y(2n^2-1)}\overline\vartheta(\frac{y}{2\alpha})|\sin(2n\pi x)|\\
&\leq\sin(2\pi x)
\Big(
\sqrt2\sum_{n=4}^\infty n^2e^{-\alpha\pi y(n^2-1)}\overline\vartheta(\frac{y}{\alpha})
+\beta\sum_{n=3}^\infty n^2e^{-\alpha\pi y(2n^2-1)}\overline\vartheta(\frac{y}{2\alpha})
\Big)\\
&\leq\sin(2\pi x)
\Big(
\sigma_3\cdot\overline\vartheta(\frac{y}{\alpha})+\sigma_4\cdot\overline\vartheta(\frac{y}{2\alpha})
\Big).
\endaligned\end{equation}
Here $\sigma_3$ and $\sigma_4$ are defined at (\ref{sigma34}).

Combining all the estimates we get
 \begin{equation}\nonumber\aligned
\mathcal{E}_{\alpha,\beta,x}(z)
&=\mathcal{E}^{a,2}_{\alpha,\beta,x}(z)+\mathcal{E}^{b,2}_{\alpha,\beta,x}(z)\\
&\geq
\sin(2\pi x)
 \Big(
 \sqrt2\underline\vartheta(\frac{y}{\alpha})
 -(\beta+\sigma_4) e^{-\alpha\pi y}\overline\vartheta(\frac{y}{2\alpha})-(4\sqrt2e^{-3\alpha\pi y}+\sigma_3)\overline\vartheta(\frac{y}{\alpha})
 \Big).
\endaligned\end{equation}

\end{proof}

\subsection{The estimates of the lower bound of a useful function and completing of the proof}
\vskip0.2in
In view of Lemma \ref{Lemma33} and \ref{Lemma34}. We
define
 \begin{equation}\nonumber\aligned
\mathcal{R}_{\alpha,\beta}(y):
=
 \Big(
 \sqrt2\underline\vartheta(\frac{y}{\alpha})
 -\beta e^{-\alpha\pi y}\overline\vartheta(\frac{y}{2\alpha})-4\sqrt2e^{-3\alpha\pi y}\overline\vartheta(\frac{y}{\alpha})
 \Big).
\endaligned\end{equation}

\begin{lemma}\label{Rab} For $\beta<\beta_0$, and $\forall\alpha\geq1, y\geq\frac{\sqrt3}{2}$,
$$
\mathcal{R}_{\alpha,\beta}(y)>0.
$$

\end{lemma}

\begin{proof} We divide its proof to three cases.

\noindent
 {\bf Case A: $\frac{y}{\alpha},\frac{y}{2\alpha}\in(0,\frac{1}{4}]$.} In this case, $\alpha\geq4y\geq2\sqrt3$.
\begin{equation}\nonumber\aligned
\mathcal{R}_{\alpha,\beta}(y)&
=
 \Big(
 \sqrt2\underline\vartheta(\frac{y}{\alpha})
 -\beta e^{-\alpha\pi y}\overline\vartheta(\frac{y}{2\alpha})-4\sqrt2e^{-3\alpha\pi y}\overline\vartheta(\frac{y}{\alpha})
 \Big)\\
 &\geq
 \sqrt2\pi e^{-\frac{\pi\alpha}{4y}}(\frac{y}{\alpha})^{-\frac{3}{2}}-\beta e^{-\pi\alpha y}(\frac{y}{2\alpha})^{-\frac{3}{2}}
 -4\sqrt2 e^{-3\pi\alpha y}(\frac{y}{\alpha})^{-\frac{3}{2}}\\
 &=
 e^{-\frac{\pi\alpha}{4y}}(\frac{y}{\alpha})^{-\frac{3}{2}}
 \Big(
 \sqrt2\pi-2\sqrt2\beta e^{-\pi\alpha(y-\frac{1}{4y})}-4\sqrt2e^{-\pi\alpha(3y-\frac{1}{4y})}
 \Big)
\endaligned\end{equation}
Trivially $y-\frac{1}{4y}\geq\frac{\sqrt3}{3}$ and $3y-\frac{1}{4y}\geq\frac{4\sqrt3}{3}$ since $y\geq\frac{\sqrt3}{2}$.
Hence
\begin{equation}\nonumber\aligned
\mathcal{R}_{\alpha,\beta}(y)
 &\geq
 e^{-\frac{\pi\alpha}{4y}}(\frac{y}{\alpha})^{-\frac{3}{2}}
 \Big(
 \sqrt2\pi-2\sqrt2\beta e^{-\pi\alpha(y-\frac{1}{4y})}-4\sqrt2e^{-\pi\alpha(3y-\frac{1}{4y})}
 \Big)\\
 &\geq
 e^{-\frac{\pi\alpha}{4y}}(\frac{y}{\alpha})^{-\frac{3}{2}}
 \Big(
 \sqrt2\pi-2\sqrt2\beta e^{-\pi\alpha\frac{\sqrt3}{3}}-4\sqrt2 e^{-\pi\alpha\frac{4\sqrt3}{3}}
 \Big)\\
 &\geq
 e^{-\frac{\pi\alpha}{4y}}(\frac{y}{\alpha})^{-\frac{3}{2}}
 \Big(
 \sqrt2\pi-2\sqrt2\beta e^{-2\pi}-4\sqrt2 e^{-8\pi}
 \Big)\\
 &>0\;\;\hbox{if}\;\; \beta<\frac{\pi e^{2\pi}-4e^{-6\pi}}{2}.
\endaligned\end{equation}

\noindent
 {\bf Case B: $\frac{y}{\alpha},\frac{y}{2\alpha}\in[\frac{1}{4},\infty)$.} In this case, $\frac{y}{\alpha}\geq\frac{1}{2}$ and there holds

\begin{equation}\nonumber\aligned
\mathcal{R}_{\alpha,\beta}(y)&
=
 \Big(
 \sqrt2\underline\vartheta(\frac{y}{\alpha})
 -\beta e^{-\alpha\pi y}\overline\vartheta(\frac{y}{2\alpha})-4\sqrt2e^{-3\alpha\pi y}\overline\vartheta(\frac{y}{\alpha})
 \Big)\\
 &\geq\Big(
 4\sqrt2\pi e^{-\pi\frac{y}{\alpha}}\big(1-\mu(\frac{y}{\alpha})\big)
 -4\pi\beta e^{-\pi\alpha y}e^{-\pi\frac{y}{2\alpha}}\big(1+\mu(\frac{y}{2\alpha})\big)\\
 &\;\;\;\;\;\;\;-16\sqrt2\pi\beta e^{-3\pi\alpha y}e^{-\pi\frac{y}{\alpha}}\big(1+\mu(\frac{y}{\alpha})\big)
 \Big)\\
 &=4\pi e^{-\frac{y}{\alpha}}
 \Big(
 \sqrt2\big(1-\mu(\frac{y}{\alpha})\big)-\beta e^{-\pi y(\alpha-\frac{1}{2\alpha})}\big(1+\mu(\frac{y}{2\alpha})\big)
 -4\sqrt2 e^{-3\pi\alpha y}\big(1+\mu(\frac{y}{\alpha})\big)
 \Big)
\endaligned\end{equation}
Trivially $e^{-\pi y(\alpha-\frac{1}{2\alpha})}\leq e^{-\frac{\sqrt3\pi}{4}}$, $\mu(\frac{y}{\alpha})\leq\mu(\frac{1}{2})$ and $\mu(\frac{y}{2\alpha})\leq\mu(\frac{1}{4})$. Hence we have

\begin{equation}\nonumber\aligned
\mathcal{R}_{\alpha,\beta}(y)
 &\geq4\pi e^{-\frac{y}{\alpha}}
 \Big(
 \sqrt2\big(1-\mu(\frac{y}{\alpha})\big)-\beta e^{-\pi y(\alpha-\frac{1}{2\alpha})}\big(1+\mu(\frac{y}{2\alpha})\big)
 -4\sqrt2 e^{-3\pi\alpha y}\big(1+\mu(\frac{y}{\alpha})\big)
 \Big)\\
 &\geq
 4\pi e^{-\frac{y}{\alpha}}
 \Big(
 \sqrt2\big(1-\mu(\frac{1}{2})\big)-\beta e^{-\frac{\sqrt3\pi}{4}}\big(1+\mu(\frac{1}{4})\big)
 -4\sqrt2 e^{-\frac{3\sqrt3\pi}{2}}\big(1+\mu(\frac{1}{2})\big)
 \Big)\\
 &>0\;\;\hbox{if}\;
 \beta<\frac{\sqrt2 e^{\frac{\sqrt3\pi}{4}}(1-\mu(\frac{1}{2}))-4e^{-\frac{5\sqrt3\pi}{4}}(1+\mu(\frac{1}{4}))}{1+\mu(\frac{1}{4})}.
\endaligned\end{equation}

\noindent
 {\bf Case C: $\frac{y}{\alpha}\in[\frac{1}{4},\infty),\frac{y}{2\alpha}\in(0,\frac{1}{4}]$.} In this case, $2y\leq\alpha\leq4y$ and we have

\begin{equation}\nonumber\aligned
\mathcal{R}_{\alpha,\beta}(y)&
=
 \Big(
 \sqrt2\underline\vartheta(\frac{y}{\alpha})
 -\beta e^{-\alpha\pi y}\overline\vartheta(\frac{y}{2\alpha})-4\sqrt2e^{-3\alpha\pi y}\overline\vartheta(\frac{y}{\alpha})
 \Big)\\
 &\geq
 4\sqrt2\pi e^{-\pi\frac{y}{\alpha}}\big(1-\mu(\frac{y}{\alpha})\big)-\beta e^{-\pi\alpha y}(\frac{y}{2\alpha})^{-\frac{3}{2}}
 -16\sqrt2\pi e^{-3\pi\alpha y}e^{-\pi\frac{y}{\alpha}}\big(1+\mu(\frac{y}{\alpha})\big)\\
 &=e^{-\pi\frac{y}{\alpha}}
 \Big(
 4\sqrt2\pi\big(1-\mu(\frac{y}{\alpha})\big)-\beta e^{-\pi y(\alpha-\frac{1}{\alpha})}(\frac{y}{2\alpha})^{-\frac{3}{2}}
 -16\sqrt2\pi e^{-3\pi\alpha y}\big(1+\mu(\frac{y}{\alpha})\big)
 \Big).
\endaligned\end{equation}
Trivially $y(\alpha-\frac{1}{\alpha})\geq1$ and $(\frac{y}{2\alpha})^{-\frac{3}{2}}\leq64$, then

\begin{equation}\nonumber\aligned
\mathcal{R}_{\alpha,\beta}(y)
 &\geq e^{-\pi\frac{y}{\alpha}}
 \Big(
 4\sqrt2\pi\big(1-\mu(\frac{y}{\alpha})\big)-\beta e^{-\pi y(\alpha-\frac{1}{\alpha})}(\frac{y}{2\alpha})^{-\frac{3}{2}}
 -16\sqrt2\pi e^{-3\pi\alpha y}\big(1+\mu(\frac{y}{\alpha})\big)
 \Big)\\
 &\geq e^{-\pi\frac{y}{\alpha}}
 \Big(
 4\sqrt2\pi\big(1-\mu(\frac{1}{4})\big)-64 e^{-\pi}\beta-16\sqrt2\pi e^{-\frac{9\pi}{2}}\big(1+\mu(\frac{1}{4})\big)
 \Big)\\
 &>0\;\;\hbox{if}\;\;\beta<\frac{4\sqrt2\pi e^\pi\big(1-\mu(\frac{1}{4})\big)-16\sqrt2\pi e^{-\frac{7\pi}{2}}\big(1+\mu(\frac{1}{4})\big)}{64}.
\endaligned\end{equation}

\end{proof}

Finally we complete the proof of Theorem \ref{2Th3}.

\begin{proof}
Combining Lemma \ref{Rab} with Lemmas \ref{Lemma33} and \ref{Lemma34}, yield that
 \begin{equation}\label{Positive}\aligned
\mathcal{E}_{\alpha,\beta,x}(z)
\begin{cases}
>0,\;\;\hbox{for}\;\; x\in(0,\frac{1}{3}), \alpha\geq1, y\geq\frac{\sqrt3}{2},\\
>0,\;\;\hbox{for}\;\; x\in[\frac{1}{3},\frac{1}{2}), \alpha\geq1, y\geq\frac{\sqrt3}{2},
\end{cases}
\endaligned\end{equation}
where $\mathcal{E}_{\alpha,\beta,x}(z)$ is defined in Lemma \ref{Lemma32}.

Since trivially
 \begin{equation}\label{Positive2}\aligned
{\mathcal{D}_{\mathcal{G}}}\subseteq\Big(\{z|: x\in[0,\frac{1}{3}], y\geq\frac{\sqrt3}{2}\}\cup\{z|: x\in[\frac{1}{3},\frac{1}{2}], y\geq\frac{\sqrt3}{2}\}\Big).
\endaligned\end{equation}
\eqref{Positive} and \eqref{Positive2} yield
\begin{equation}\label{Positive3}\aligned
\mathcal{E}_{\alpha,\beta,x}(z)
>0,\;\;\hbox{for}\;\; \alpha\geq1\;\;\hbox{and}\;\;z\in {\mathcal{D}_{\mathcal{G}}}.
\endaligned\end{equation}
Therefore, \eqref{Positive3} and Lemma \ref{Lemma32} complete the proof of Theorem \ref{2Th3}.

\end{proof}

\section{The monotonicity on the vertical line $y=\frac{1}{2}$ and proof of Theorem \ref{Th1}}
\setcounter{equation}{0}

In Theorem \ref{2Th1}, we have established that for $\alpha\geq1, \beta<\beta_0\cong3.8$,
\begin{equation}\aligned\label{3hhh}
\min_{z\in\mathbb{H}}\Big(\theta(\alpha;z)-\beta\theta(2\alpha;z)\Big)
=\min_{z\in\Gamma}\Big(\theta(\alpha;z)-\beta\theta(2\alpha;z)\Big),
\endaligned\end{equation}
where $\Gamma$ is a vertical line and defined at \eqref{Gfff}.

The following Lemma, which proves the non-existence part of Theorem \ref{Th1},  shows  that one only needs to consider the minimum of $\Big(\theta(\alpha;z)-\beta\theta(2\alpha;z)\Big)$
for $\beta\leq\sqrt2$.
\begin{lemma}\label{Lemma41} The minimum of $\Big(\theta(\alpha;z)-\beta\theta(2\alpha;z)\Big)$ does not exist if $\beta>\sqrt2$.
\end{lemma}

We postpone the proof of Lemma \ref{Lemma41} to the late section.  Combining \eqref{3hhh} of  Theorem \ref{2Th1} and Lemma \ref{Lemma41}, it suffices to consider minimization problem on a vertical line $\Gamma$ for $\beta\leq\sqrt2$. For this, we establish the following, which proves the first part of Theorem \ref{Th1}.

\begin{theorem}\label{4Th1} Assume that $\alpha\geq1$. For $\beta\leq\sqrt2$,
\begin{equation}\aligned\nonumber
\min_{z\in\Gamma}\Big(\theta(\alpha;z)-\beta\theta(2\alpha;z)\Big)\;\;\hbox{is achieved uniquely at }\;\;\frac{1}{2}+i\frac{\sqrt3}{2}.
\endaligned\end{equation}
\end{theorem}

Since $\frac{\partial}{\partial y}\theta(\alpha;\frac{1}{2}+i y)\geq0$ for $y\geq\frac{\sqrt3}{2}$ with equality attained at $\frac{\sqrt3}{2}$ $($\cite{Mon1988}$)$, it suffices to prove
the critical case of Theorem \ref{4Th1}, namely,
\begin{theorem}\label{4Th2} Assume that $\alpha\geq1$ and $\beta=\sqrt{2}$. Then
\begin{equation}\aligned\nonumber
\min_{z\in\Gamma}\Big(\theta(\alpha;z)-\sqrt2\theta(2\alpha;z)\Big)\;\;\hbox{is achieved uniquely at }\;\;\frac{1}{2}+i\frac{\sqrt3}{2}.
\endaligned\end{equation}
\end{theorem}

In the rest of this section, we aim to prove Theorem \ref{4Th2} (which is a consequence of  Lemma \ref{Lemma44} and \ref{Lemma46}). Due to its difficulty and complexity, we shall divide its proof into two cases.

We first establish the following
\begin{lemma}\label{Lemma42} Assume that $\alpha\geq1$ and $y\geq\frac{\sqrt3}{2}$.
 Then we have the following  lower bound  estimate
 \begin{equation}\aligned\nonumber
\big(\frac{\partial^2}{\partial y^2}+\frac{2}{y}\frac{\partial}{\partial y}\big)\Big(\theta(\alpha;\frac{1}{2}+iy)-\sqrt2\theta(2\alpha;\frac{1}{2}+iy)\Big)
\geq\frac{2(\pi\alpha)^2}{y^4}e^{-\pi\frac{\alpha}{y}}\cdot
\mathcal{W}(y;\alpha),
\endaligned\end{equation}
where
\begin{equation}\aligned\label{www111}
\mathcal{W}(y;\alpha):=1+(2(y^2-\frac{1}{4})^2-\frac{4}{\pi\alpha}y^3\cdot(1+\epsilon_{a}))\cdot e^{-\pi\alpha (y-\frac{3}{4y})}
-4\sqrt2(1+\epsilon_{b})\cdot e^{-\pi\frac{\alpha}{y}}.
\endaligned\end{equation}
Here $\epsilon_a$ is small and can be explicitly controlled by
$$
\epsilon_a:=\epsilon_{a,1}+\epsilon_{a,2}+\epsilon_{a,3}+\epsilon_{a,4}.
$$
and
\begin{equation}\aligned\nonumber
\epsilon_a\rightarrow0\;\;\hbox{as}\;\;y\mapsto\infty.
\endaligned\end{equation}

Here each $\epsilon_{a,j}(j=1,2,3,4)$ is small and expressed by
\begin{equation}\aligned\nonumber
\epsilon_{a,1}:&=\sum_{n=2}^\infty(2n-1)^2 e^{-\pi\alpha y((2n-1)^2-1)}\\
\epsilon_{a,2}:&=\sum_{n=2}^\infty e^{-\frac{\pi\alpha}{4y}((2n-1)^2-1)}\\
\epsilon_{a,3}:&=\epsilon_{a,1}\cdot\epsilon_{a,2}\\
\epsilon_{a,4}:&=2 e^{-\pi\alpha(3y-\frac{1}{4y})}\big(1+\sum_{n=2}^\infty n^2 e^{-4\pi\alpha y(n^2-1)}\big)\cdot\vartheta_3(\frac{\alpha}{y})
.
\endaligned\end{equation}
And $\epsilon_b$ is small and consist of four smaller parts
\begin{equation}\aligned\nonumber
\epsilon_b:=\epsilon_{b,1}+\epsilon_{b,2}+\epsilon_{b,3}+\epsilon_{b,4},
\endaligned\end{equation}
and
\begin{equation}\aligned\nonumber
\epsilon_b\rightarrow0\;\;\hbox{as}\;\;y\mapsto\infty.
\endaligned\end{equation}

Here
\begin{equation}\aligned\nonumber
\epsilon_{b,1}
:&=2y^4 e^{-2\pi\alpha y}\cdot(1+\sum_{n=2}^\infty e^{-\frac{2\pi\alpha}{y}((2n-1)^2-1)})\cdot
(1+\sum_{n=2}^\infty(2n-1)^4 e^{-2\pi\alpha y ((2n-1)^2-1)})\\
\epsilon_{b,2}
:&=\frac{1}{8}e^{-2\pi\alpha y}\cdot(1+\sum_{n=2}^\infty (2n-1)^4 e^{-\frac{2\pi\alpha}{y}((2n-1)^2-1)})
\cdot(1+\sum_{n=2}^\infty e^{-2\pi\alpha y((2n-1)^2-1)})\\
\epsilon_{b,3}
:&=16y^4 e^{-\pi \alpha (8y-\frac{2}{y})}\cdot(1+\sum_{n=2}^\infty n^4e^{-8\pi\alpha y(n^2-1)})\cdot(1+2\sum_{n=1}^\infty e^{-2\pi\frac{\alpha}{y}n^2})\\
\epsilon_{b,4}
:&=y^4 e^{-\pi \alpha (8y-\frac{2}{y})}\cdot(1+\sum_{n=2}^\infty e^{-8\pi\alpha y(n^2-1)})\cdot(1+2\sum_{n=1}^\infty \frac{n^4}{y^4}e^{-2\pi\frac{\alpha}{y}n^2}).
\endaligned\end{equation}

\end{lemma}

An elementary estimate of $\mathcal{W}(y;\alpha)$ in Lemma \ref{Lemma42}, we obtain
\begin{lemma}\label{Lemma43} Assume that $\alpha\geq1$. If $y\in[\frac{\sqrt3}{2},1.8\alpha]$,

\begin{equation}\aligned\nonumber
\big(\frac{\partial^2}{\partial y^2}+\frac{2}{y}\frac{\partial}{\partial y}\big)\Big(\theta(\alpha;\frac{1}{2}+iy)-\sqrt2\theta(2\alpha;\frac{1}{2}+iy)\Big)
>0.
\endaligned\end{equation}
\end{lemma}

In view of Lemma \ref{Lemma43}, one has
\begin{lemma}\label{Lemma44} Assume that $\alpha\geq1$. If $y\in[\frac{\sqrt3}{2},1.8\alpha]$,

\begin{equation}\aligned\nonumber
\frac{\partial}{\partial y}\Big(\theta(\alpha;\frac{1}{2}+iy)-\sqrt\theta(2\alpha;\frac{1}{2}+iy)\Big)
>0.
\endaligned\end{equation}
\end{lemma}

\begin{proof} The proof follows from Lemma \ref{Lemma43}. In fact
it suffices to notice that
\begin{equation}\aligned\nonumber
&\frac{\partial}{\partial y}\Big(y^{-2}\Big(
\frac{\partial}{\partial y}\big(\theta(\alpha;\frac{1}{2}+iy)-\sqrt\theta(2\alpha;\frac{1}{2}+iy)\big)
\Big)\Big)\\
=&\big(\frac{\partial^2}{\partial y^2}+\frac{2}{y}\frac{\partial}{\partial y}\big)\Big(\theta(\alpha;\frac{1}{2}+iy)-\sqrt2\theta(2\alpha;\frac{1}{2}+iy)\Big)
\endaligned\end{equation}
and
\begin{equation}\aligned\nonumber
\Big(
\frac{\partial}{\partial y}\big(\theta(\alpha;\frac{1}{2}+iy)-\sqrt\theta(2\alpha;\frac{1}{2}+iy)\big)
\Big)\mid_{y=\frac{\sqrt3}{2}}=0.
\endaligned\end{equation}
\end{proof}


\begin{lemma}\label{Lemma45} The following estimates hold
\begin{equation}\aligned\nonumber
\frac{\partial}{\partial y}\Big(\theta (\alpha;\frac{1}{2}+iy)-\sqrt2\theta (2\alpha;\frac{1}{2}+iy)
\Big)&\geq
\frac{2}{\sqrt{y\alpha}}e^{-\pi\frac{y}{2\alpha}}\cdot \mathcal{Q}(y;\alpha),
\endaligned\end{equation}
where
\begin{equation}\aligned\nonumber
\mathcal{Q}(y;\alpha):=\frac{\pi y}{2\alpha}-\frac{1}{2}-(\frac{\pi}{\alpha}-\frac{1}{2}) e^{-\pi\frac{y}{2\alpha}}
-ye^{-\pi y(\alpha-\frac{1}{2\alpha})}\mathcal{P}(y;\alpha)
-ye^{-\pi y(2\alpha-\frac{1}{2\alpha})}\mathcal{P}(y;2\alpha).
\endaligned\end{equation}

\begin{equation}\aligned\nonumber
\mathcal{P}(y;\alpha):=\sigma_5+\sigma_6+\sigma_7.
\endaligned\end{equation}
Here
\begin{equation}\aligned\nonumber
\sigma_5:&=
\frac{1}{2y}\big(1+2e^{-\pi\frac{y}{\alpha}}(1+\nu(\frac{y}{\alpha}))\big)(1+\nu(y\alpha))\\
\sigma_6:&=\alpha\pi(1+\mu(y\alpha))(1+2e^{-\pi\frac{y}{\alpha}}(1+\nu(\frac{y}{\alpha})))\\
\sigma_7:&=\frac{2\pi}{\alpha}e^{-\pi\frac{y}{\alpha}}(1+\nu(y\alpha))(1+\mu(\frac{y}{\alpha})).
\endaligned\end{equation}

\begin{equation}\aligned\nonumber
\mu(y):&=\sum_{n=2}^\infty n^2 e^{-\pi y(n^2-1)}\\
\nu(y):&=\sum_{n=2}^\infty  e^{-\pi y(n^2-1)}
\endaligned\end{equation}

\end{lemma}
Based on Lemma \ref{Lemma45}, an elementary analysis of $\mathcal{Q}(y;\alpha)$ yields

\begin{lemma}\label{Lemma46} Assume that $\alpha\geq1$. If $y\in[1.15\alpha,\infty]$,

\begin{equation}\aligned\nonumber
\frac{\partial}{\partial y}\Big(\theta (\alpha;\frac{1}{2}+iy)-\sqrt2\theta (2\alpha;\frac{1}{2}+iy)
\Big)>0.
\endaligned\end{equation}
\end{lemma}
Trivially, for $\alpha\geq1$,
\begin{equation}\aligned\nonumber
[\frac{\sqrt3}{2},\infty)\subseteq[\frac{\sqrt3}{2}, 1.8\alpha]\cup[1.15\alpha,\infty)
\endaligned\end{equation}
then Lemmas \ref{Lemma46} and \ref{Lemma44} complete the proof of Theorem \ref{4Th2}.

In the rest of this Section, we provide the proof of Lemmas \ref{Lemma46} and \ref{Lemma44}.
\subsection{Some basic identities}

\begin{lemma}\label{Lemma47} The following identity for $\big(\frac{\partial^2}{\partial y^2}+\frac{2}{y}\frac{\partial}{\partial y}\big)\big(\theta(\alpha;z)-\sqrt2\theta(2\alpha;z)\big)$ holds
 \begin{equation}\nonumber\aligned
\big(\frac{\partial^2}{\partial y^2}+\frac{2}{y}\frac{\partial}{\partial y}\big)&\big(\theta(\alpha;z)-\sqrt2\theta(2\alpha;z)\big)=(\pi\alpha)^2\sum_{n,m} (n^2-\frac{(m+nx)^2}{y^2})^2e^{-\pi\alpha(yn^2+\frac{(m+nx)^2}{y})}\\
&+\frac{4\sqrt2\pi\alpha}{y}\sum_{n,m}n^2e^{-2\pi\alpha(yn^2+\frac{(m+nx)^2}{y})}
-\frac{2\pi\alpha}{y}\sum_{n,m}n^2e^{-\pi\alpha(yn^2+\frac{(m+nx)^2}{y})}\\
&-4\sqrt2(\pi\alpha)^2\sum_{n,m} (n^2-\frac{(m+nx)^2}{y^2})^2e^{-2\pi\alpha(yn^2+\frac{(m+nx)^2}{y})}
\endaligned\end{equation}
\end{lemma}
\begin{proof} By definition of the theta function (\ref{thetas}), we have
 \begin{equation}\nonumber\aligned
\frac{\partial}{\partial y}\theta(\alpha;z)=\pi\alpha\sum_{n,m} n^2e^{-\pi\alpha(yn^2+\frac{(m+nx)^2}{y})}
-\pi\alpha\sum_{n,m} \frac{(m+nx)^2}{y^2}e^{-\pi\alpha(yn^2+\frac{(m+nx)^2}{y})}
\endaligned\end{equation}
and
 \begin{equation}\label{Th471}\aligned
\big(\frac{\partial^2}{\partial y^2}+\frac{2}{y}\frac{\partial}{\partial y}\big)\theta(\alpha;z)=&(\pi\alpha)^2\sum_{n,m} (n^2-\frac{(m+nx)^2}{y^2})^2e^{-\pi\alpha(yn^2+\frac{(m+nx)^2}{y})}
\\&-\frac{2\pi\alpha}{y}\sum_{n,m} n^2e^{-\pi\alpha(yn^2+\frac{(m+nx)^2}{y})}.
\endaligned\end{equation}
The identity follows by \eqref{Th471}.

\end{proof}

Similar to the proof of Lemma \ref{Lemma47} (the details are omit here), one has
\begin{lemma} It holds
 \begin{equation}\nonumber\aligned
\frac{\partial}{\partial y}\big(\theta(\alpha;z)-\sqrt2\theta(2\alpha;z)\big)&=\pi\alpha\sum_{n,m}\frac{(m+nx)^2}{y^2}e^{-\pi\alpha(yn^2+\frac{(m+nx)^2}{y})}
+2\sqrt2\pi\alpha\sum_{n,m}n^2e^{-2\pi\alpha(yn^2+\frac{(m+nx)^2}{y})}\\
&-\pi\alpha\sum_{n,m}n^2e^{-\pi\alpha(yn^2+\frac{(m+nx)^2}{y})}
-2\sqrt2\pi\alpha\sum_{n,m}\frac{(m+nx)^2}{y^2}e^{-2\pi\alpha(yn^2+\frac{(m+nx)^2}{y})}
\endaligned\end{equation}

\end{lemma}

\subsection{The analysis of $\frac{\partial}{\partial y}\Big(\theta(\alpha;z)-\sqrt2 \theta(2\alpha;z)\Big)$}

 We use the following expression of theta function, which is a variant of Lemma \ref{Lemma31}.
\begin{lemma}\label{Lemma409} A variant expression of $\theta (\alpha;z)$ is the following
\begin{equation}\aligned\label{PPP3}
\theta (\alpha;z)&=2\sqrt{\frac{y}{\alpha}}\sum_{n=1}^\infty e^{-\alpha \pi yn^2}\vartheta(\frac{y}{\alpha};nx)+\sqrt{\frac{y}{\alpha}}\vartheta_3(\frac{y}{\alpha})\\
&=\sqrt{\frac{y}{\alpha}}\big(1+2\sum_{n=1}^\infty e^{-n^2\pi\frac{y}{\alpha}}\big)+2\sqrt{\frac{y}{\alpha}}\sum_{n=1}^\infty e^{-\alpha \pi y n^2}\vartheta(\frac{y}{\alpha};nx).
\endaligned\end{equation}

\end{lemma}

Now we give the proof of Lemma \ref{Lemma41}:

\begin{proof} In view of Lemma \ref{Lemma409}, one has
\begin{equation}\aligned\nonumber
\theta (\alpha;z)&=\sqrt{\frac{y}{\alpha}}\cdot\big(
1+2e^{-\pi\frac{y}{\alpha}}+2e^{-\alpha\pi y}+o(e^{-\pi\frac{y}{\alpha}})+o(e^{-\alpha\pi y})
\big).
\endaligned\end{equation}
Then
\begin{equation}\aligned\nonumber
\theta (\alpha;z)-\beta\theta(2\alpha;z)
&=\sqrt{\frac{y}{2\alpha}}\cdot\Big(
\sqrt2-\beta+2\sqrt2 e^{-\pi\alpha y}-2e^{-\pi\frac{y}{2\alpha}}+o(e^{-\pi\alpha y})+o(e^{-\pi\frac{y}{2\alpha}})
\Big),\\
&=\sqrt{\frac{y}{2\alpha}}\cdot\Big(
\sqrt2-\beta+o(1)
\Big)
\endaligned\end{equation}
Therefore, for $\forall\alpha>0$,
\begin{equation}\aligned\nonumber
\theta (\alpha;z)-\beta\theta(2\alpha;z)
&=\sqrt{\frac{y}{2\alpha}}\cdot\Big(
\sqrt2-\beta+o(1)
\Big)\\
&\mapsto -\infty,\;\;\;\hbox{if}\;\;\;\beta>\sqrt2,
\endaligned\end{equation}
proves the nonexistence result.

\end{proof}

In the next two lemmas(Lemmas \ref{Lemma410} and \ref{Lemma411}), we analyze the two parts of $\theta(\alpha;z)$ in Lemma \ref{Lemma409}.

\begin{lemma}[Analysis of second part arising from Lemma \ref{Lemma409}]\label{Lemma410}Assume that $\alpha\geq1$. If $\frac{y}{\alpha}\geq\frac{4}{5}$, then
\begin{equation}\aligned\nonumber
\frac{\partial}{\partial y}\Big(\sqrt{y}\Big(
\vartheta_3(\frac{y}{\alpha})-\vartheta_3(\frac{y}{2\alpha})
\Big)\Big)>0.
\endaligned\end{equation}

\end{lemma}

\begin{proof} By a straightforward computation, we have

\begin{equation}\aligned\label{gggjjj}
&\frac{\partial}{\partial y}\Big(\sqrt{y}\Big(
\vartheta_3(\frac{y}{\alpha})-\vartheta_3(\frac{y}{2\alpha})
\Big)\Big)\\
=&2\frac{\partial}{\partial y}\Big(\sqrt{y}
\sum_{n=1}^\infty \big(e^{-n^2\pi\frac{y}{\alpha}}-e^{-n^2\pi\frac{y}{2\alpha}}\big)
\Big)\\
=&
\frac{2}{\sqrt{y}}\sum_{n=1}^\infty e^{-n^2\pi\frac{y}{\alpha}}\Big(
(\frac{n^2\pi y}{2\alpha}-\frac{1}{2})e^{n^2\pi \frac{y}{2\alpha}}-(\frac{n^2\pi }{\alpha}-\frac{1}{2})
\Big)
\endaligned\end{equation}
Since $\alpha\geq1,\;\;\frac{y}{\alpha}\geq\frac{4}{5}$, then
\begin{equation}\aligned\nonumber
(\frac{\pi y}{2\alpha}-\frac{1}{2})e^{\pi \frac{y}{2\alpha}}-(\frac{\pi }{\alpha}-\frac{1}{2})
&\geq(\frac{\pi y}{2\alpha}-\frac{1}{2})e^{\pi \frac{y}{2\alpha}}-(\pi-\frac{1}{2})\\
&>0.
\endaligned\end{equation}
Therefore, since each term in the sum of \eqref{gggjjj} is positive, the result then follows.
\end{proof}

To control the error terms, we recall
\begin{equation}\aligned\nonumber
\mu(y):&=\sum_{n=2}^\infty n^2 e^{-\pi y(n^2-1)}\\
\nu(y):&=\sum_{n=2}^\infty  e^{-\pi y(n^2-1)}
\endaligned\end{equation}

\begin{lemma} [The estimate of first part in Lemma \ref{Lemma409}]\label{Lemma411}

\begin{equation}\aligned\nonumber
\frac{\partial}{\partial y}\Big(\sqrt{y}\sum_{n=1}^\infty e^{-\alpha \pi y n^2}\vartheta(\frac{y}{\alpha};nx)\Big)
\leq
\sqrt{y}e^{-\pi y\alpha}\cdot \mathcal{P}(y;\alpha)
\endaligned\end{equation}
where the $\mathcal{P}(y;\alpha)$ can be controlled by some constant and is expressed by

\begin{equation}\aligned\nonumber
\mathcal{P}(y;\alpha):=\sigma_5+\sigma_6+\sigma_7,
\endaligned\end{equation}
and
\begin{equation}\aligned\nonumber
\mathcal{P}(y;\alpha)<\alpha\pi,\;\hbox{as}\;\; y\rightarrow\infty.
\endaligned\end{equation}

Here
\begin{equation}\aligned\nonumber
\sigma_5:&=
\frac{1}{2y}\big(1+2e^{-\pi\frac{y}{\alpha}}(1+\nu(\frac{y}{\alpha}))\big)(1+\nu(y\alpha))\\
\sigma_6:&=\alpha\pi(1+\mu(y\alpha))(1+2e^{-\pi\frac{y}{\alpha}}(1+\nu(\frac{y}{\alpha})))\\
\sigma_7:&=\frac{2\pi}{\alpha}e^{-\pi\frac{y}{\alpha}}(1+\nu(y\alpha))(1+\mu(\frac{y}{\alpha})).
\endaligned\end{equation}

\end{lemma}
\begin{remark} Lemma \ref{Lemma411} shows that
$$\frac{\partial}{\partial y}\Big(\sqrt{y}\sum_{n=1}^\infty e^{-\alpha \pi y n^2}\vartheta(\frac{y}{\alpha};nx)\Big)$$
is small in related estimates.

\end{remark}

\begin{proof}  A direct calculation shows that
\begin{equation}\aligned\nonumber
\frac{\partial}{\partial y}\Big(\sqrt{y}\sum_{n=1}^\infty e^{-\alpha \pi y n^2}\vartheta(\frac{y}{\alpha};nx)\Big)
=&\frac{1}{2\sqrt{y}}\sum_{n=1}^\infty e^{-\alpha \pi y n^2}\vartheta(\frac{y}{\alpha};nx)
+\sqrt{y}\sum_{n=1}^\infty(-\alpha\pi y) e^{-\alpha \pi y n^2}\vartheta(\frac{y}{\alpha};nx)\\
&+\sqrt{y}\sum_{n=1}^\infty(-\alpha\pi y) e^{-\alpha \pi y n^2}\frac{1}{\alpha}\frac{\partial}{\partial X}\vartheta(\frac{y}{\alpha};nx).
\endaligned\end{equation}
For convenience, we denote that
\begin{equation}\aligned\nonumber
I_1:=&\frac{1}{2\sqrt{y}}\sum_{n=1}^\infty e^{-\alpha \pi y n^2}\vartheta(\frac{y}{\alpha};nx)\\
I_2:=&\sqrt{y}\sum_{n=1}^\infty(-\alpha\pi n^2) e^{-\alpha \pi y n^2}\vartheta(\frac{y}{\alpha};nx)\\
I_3:=&\sqrt{y}\sum_{n=1}^\infty e^{-\alpha \pi y n^2}\frac{1}{\alpha}\frac{\partial}{\partial X}\vartheta(\frac{y}{\alpha};nx).
\endaligned\end{equation}
Then
\begin{equation}\aligned\label{4I123}
\frac{\partial}{\partial y}\Big(\sqrt{y}\sum_{n=1}^\infty e^{-\alpha \pi y n^2}\vartheta(\frac{y}{\alpha};nx)\Big)
=I_1+I_2+I_3.
\endaligned\end{equation}
Next, we estimate $I_j,j=1,2,3$ in order.
For $I_1$,

\begin{equation}\aligned\label{4I1}
|I_1|&=|\frac{1}{2\sqrt{y}}\sum_{n=1}^\infty e^{-\alpha \pi y n^2}\vartheta(\frac{y}{\alpha};nx)|\\
&\leq\frac{1}{2\sqrt{y}}\sum_{n=1}^\infty e^{-\alpha \pi y n^2}|\vartheta(\frac{y}{\alpha};nx)|\\
&\leq\frac{1}{2\sqrt{y}}\sum_{n=1}^\infty e^{-\alpha \pi y n^2}(1+2\sum_{n=1}^\infty e^{-n^2\pi \frac{y}{\alpha}})\\
&=\frac{1}{2\sqrt{y}} e^{-\alpha\pi y}\big(1+\sum_{n=2}^\infty e^{-\alpha\pi y(n^2-1)}\big)\cdot\big(
1+2e^{-\pi\frac{y}{\alpha}}(1+\sum_{n=2}^\infty e^{-\pi\frac{y}{\alpha}(n^2-1)})
\big)\\
&=\frac{1}{2\sqrt{y}} e^{-\alpha\pi y}\big(
1+\nu(y\alpha)
\big)\big(
1+2 e^{-\pi\frac{y}{\alpha}}(1+\nu(\frac{y}{\alpha}))
\big).
\endaligned\end{equation}

And $I_2$,
\begin{equation}\aligned\label{4I2}
|I_2|&=|\sqrt{y}\sum_{n=1}^\infty(-\alpha\pi n^2) e^{-\alpha \pi y n^2}\vartheta(\frac{y}{\alpha};nx)|\\
&\leq\alpha\pi\sqrt{y}\sum_{n=1}^\infty n^2 e^{-\alpha \pi y n^2}|\vartheta(\frac{y}{\alpha};nx)|\\
&\leq\alpha\pi\sqrt{y}\sum_{n=1}^\infty n^2 e^{-\alpha \pi y n^2}\vartheta_3(\frac{y}{\alpha})\\
&=\alpha\pi\sqrt{y}\sum_{n=1}^\infty n^2 e^{-\alpha \pi y n^2}\big(1+2\sum_{n=1}^\infty e^{-\pi n^2\frac{y}{\alpha}}\big)\\
&=\alpha\pi\sqrt{y} e^{-\alpha\pi y}\big(1+\sum_{n=2}^\infty e^{-\alpha \pi y (n^2-1)} \big)\cdot
\big(1+2e^{-\pi\frac{y}{\alpha}}(1+\sum_{n=2}^\infty e^{-\pi (n^2-1)\frac{y}{\alpha}})\big)\\
&=\alpha\pi\sqrt{y} e^{-\alpha\pi y}\big(1+\mu(y\alpha) \big)\cdot
\big(1+2e^{-\pi\frac{y}{\alpha}}(1+\nu(\frac{y}{\alpha}))\big).
\endaligned\end{equation}

The $I_3$ is estimated by
\begin{equation}\aligned\label{4I3}
|I_3|=&|\sqrt{y}\sum_{n=1}^\infty e^{-\alpha \pi y n^2}\frac{1}{\alpha}\frac{\partial}{\partial X}\vartheta(\frac{y}{\alpha};nx)|\\
\leq&2\pi\sqrt{y}\frac{1}{\alpha}\sum_{n=1}^\infty e^{-\alpha \pi y n^2}\cdot\sum_{n=1}^\infty n^2 e^{-\pi n^2\frac{y}{\alpha}}\\
=&\frac{2}{\alpha}\pi\sqrt{y} e^{-\alpha\pi y} e^{-\pi\frac{y}{\alpha}}\big(1+\sum_{n=2}^\infty e^{-\alpha\pi y(n^2-1)}\big)
\cdot\sum_{n=2}^\infty n^2 e^{-\pi (n^2-1)\frac{y}{\alpha}}\\
=&\frac{2}{\alpha}\pi\sqrt{y} e^{-\pi y(\alpha+\frac{1}{\alpha})}\cdot(1+\nu(y\alpha))\cdot(1+\mu(\frac{y}{\alpha})).
\endaligned\end{equation}

The \eqref{4I123} togethers with \eqref{4I1}, \eqref{4I2} and \eqref{4I3} yield the result.

\end{proof}


The following lemma is a variant of Lemma \ref{Lemma45}.

\begin{lemma} \label{Lemmalower}
 Assume that $\alpha\geq1$. If $\frac{y}{\alpha}\geq\frac{4}{5}$, then it holds
\begin{equation}\aligned\nonumber
\frac{\partial}{\partial y}\Big(\theta (\alpha;z)-\sqrt2\theta (2\alpha;z)
\Big)&\geq
\frac{2}{\sqrt{y\alpha}}e^{-\pi\frac{y}{2\alpha}}\cdot \mathcal{Q}(y;\alpha),
\endaligned\end{equation}
where
\begin{equation}\aligned\nonumber
\mathcal{Q}(y;\alpha):=\frac{\pi y}{2\alpha}-\frac{1}{2}-(\frac{\pi}{\alpha}-\frac{1}{2}) e^{-\pi\frac{y}{2\alpha}}
-ye^{-\pi y(\alpha-\frac{1}{2\alpha})}\mathcal{P}(y;\alpha)
-ye^{-\pi y(2\alpha-\frac{1}{2\alpha})}\mathcal{P}(y;2\alpha),
\endaligned\end{equation}
and $\mathcal{P}(y;\alpha)$ is introduced in Lemma \ref{Lemma411}.
\end{lemma}

\begin{proof} By Lemmas \ref{Lemma409}, \ref{Lemma410} and \ref{Lemma411},
\begin{equation}\aligned\label{Lhhh}
\frac{\sqrt{\alpha}}{2}\frac{\partial}{\partial y}\Big(\theta (\alpha;z)-\sqrt2\theta (2\alpha;z)
\Big)&\geq
\frac{1}{\sqrt{y}}\sum_{n=1}^\infty e^{-n^2\pi\frac{y}{\alpha}}\Big(
(\frac{n^2\pi y}{2\alpha}-\frac{1}{2})e^{n^2\pi \frac{y}{2\alpha}}-(\frac{n^2\pi }{\alpha}-\frac{1}{2})
\Big)\\
&
-\sqrt{y} e^{-\pi y\alpha}\mathcal{P}(y;\alpha)
-\sqrt{y} e^{-2\pi y\alpha}\mathcal{P}(y;2\alpha)\\
&=\frac{1}{\sqrt{y}} e^{-\pi\frac{y}{\alpha}}\Big(
(\frac{\pi y}{2\alpha}-\frac{1}{2})e^{\pi \frac{y}{2\alpha}}-(\frac{\pi }{\alpha}-\frac{1}{2})
\Big)\\
&
-\sqrt{y} e^{-\pi y\alpha}\mathcal{P}(y;\alpha)
-\sqrt{y} e^{-2\pi y\alpha}\mathcal{P}(y;2\alpha)\\
&+\frac{1}{\sqrt{y}}\sum_{n=2}^\infty e^{-n^2\pi\frac{y}{\alpha}}\Big(
(\frac{n^2\pi y}{2\alpha}-\frac{1}{2})e^{n^2\pi \frac{y}{2\alpha}}-(\frac{n^2\pi }{\alpha}-\frac{1}{2})
\Big)\\
&=\frac{1}{\sqrt{y}}e^{-\pi\frac{y}{2\alpha}}\cdot \mathcal{Q}(y;\alpha)+\mathcal{R}_0(y;\alpha).
\endaligned\end{equation}
Here
\begin{equation}\aligned\nonumber
\mathcal{R}_0(y;\alpha):=\frac{1}{\sqrt{y}}\sum_{n=2}^\infty e^{-n^2\pi\frac{y}{\alpha}}\Big(
(\frac{n^2\pi y}{2\alpha}-\frac{1}{2})e^{n^2\pi \frac{y}{2\alpha}}-(\frac{n^2\pi }{\alpha}-\frac{1}{2})
\Big).
\endaligned\end{equation}
Since $\alpha\geq1,\;\;\frac{y}{\alpha}\geq\frac{4}{5}$,
\begin{equation}\aligned\nonumber
(\frac{\pi y}{2\alpha}-\frac{1}{2})e^{\pi \frac{y}{2\alpha}}-(\frac{\pi }{\alpha}-\frac{1}{2})
&\geq(\frac{\pi y}{2\alpha}-\frac{1}{2})e^{\pi \frac{y}{2\alpha}}-(\pi-\frac{1}{2})\\
&>0.
\endaligned\end{equation}

Then trivially,
\begin{equation}\aligned\nonumber
\mathcal{R}_0(y;\alpha)>0.
\endaligned\end{equation}
Then the estimate follows by \eqref{Lhhh}.

\end{proof}

\begin{lemma}[The upper bounds of $y\cdot\mathcal{P}(y;\alpha)$ and $y\cdot\mathcal{P}(y;2\alpha)$]\label{Lemma413a} Assume that $\alpha\geq1 , y\geq\frac{\sqrt3}{2}$. If $\frac{y}{\alpha}\geq1$, then
\begin{equation}\aligned\nonumber
y\cdot\mathcal{P}(y;\alpha)&\leq4.232412\cdots,\\
y\cdot\mathcal{P}(y;2\alpha)&\leq10.268696\cdots.
\endaligned\end{equation}
\end{lemma}

\begin{proof} In view of the expression $y\cdot\mathcal{P}(y;\alpha), y\cdot\mathcal{P}(y;2\alpha)$. The only technical part is to control
\begin{equation}\aligned\nonumber
\alpha y e^{-\pi y(\alpha-\frac{1}{2\alpha})},  \alpha y e^{-\pi y(2\alpha-\frac{1}{2\alpha})}.
\endaligned\end{equation}
These two terms are similar. We estimate by
\begin{equation}\aligned\label{hhh000}
\alpha y e^{-\pi y(\alpha-\frac{1}{2\alpha})}
&=\alpha^2\cdot\frac{y}{\alpha}e^{-\pi \alpha^2\cdot \frac{y}{\alpha}\cdot(1-\frac{1}{2\alpha^2})}\\
&\geq\frac{y}{\alpha}e^{-\pi \frac{y}{\alpha}\cdot(1-\frac{1}{2\alpha^2})},
\endaligned\end{equation}
where we shall control the growth of $\alpha$ by the monotonically decreasing of $xe^{-A\cdot x}$ as $x\geq\frac{1}{A}$.
Similar to \eqref{hhh000}
\begin{equation}\aligned\label{hhh111}
\alpha y e^{-\pi y(2\alpha-\frac{1}{2\alpha})}
\geq\frac{y}{\alpha}e^{-\pi \frac{y}{\alpha}\cdot(2-\frac{1}{2\alpha^2})}.
\endaligned\end{equation}
Then we can view $\frac{y}{\alpha}$ as an variable in estimates.
The rest of estimate using the decreasing of $\mu,\nu$. Namely,
\begin{equation}\aligned\nonumber
&\mu(y\alpha)\leq\mu(\frac{\sqrt3}{2}), \mu(y\alpha)\leq\mu(1);\\
&\nu(y\alpha)\leq\nu(\frac{\sqrt3}{2}), \nu(y\alpha)\leq\nu(1).
\endaligned\end{equation}
Here $\alpha\geq1, y\geq\frac{\sqrt3}{2}$ and $\frac{y}{\alpha}\geq1$ used.
\end{proof}

\begin{lemma}\label{Lemma4L} Assume that $\alpha\geq1$, if $\frac{y}{\alpha}\geq1.15$, then
\begin{equation}\aligned\nonumber
\mathcal{Q}(y;\alpha)>0,
\endaligned\end{equation}
where
\begin{equation}\aligned\nonumber
\mathcal{Q}(y;\alpha)=\frac{\pi y}{2\alpha}-\frac{1}{2}-(\frac{\pi}{\alpha}-\frac{1}{2}) e^{-\pi\frac{y}{2\alpha}}
-ye^{-\pi y(\alpha-\frac{1}{2\alpha})}\mathcal{P}(y;\alpha)
-ye^{-\pi y(2\alpha-\frac{1}{2\alpha})}\mathcal{P}(y;2\alpha)
\endaligned\end{equation}
 is defined in Lemma \ref{Lemmalower}.
\end{lemma}

\begin{proof} The proof follows from Lemma \ref{Lemma413a}:
\begin{equation}\aligned\label{qqq111}
\mathcal{Q}(y;\alpha)&=\frac{\pi y}{2\alpha}-\frac{1}{2}-(\frac{\pi}{\alpha}-\frac{1}{2}) e^{-\pi\frac{y}{2\alpha}}
-ye^{-\pi y(\alpha-\frac{1}{2\alpha})}\mathcal{P}(y;\alpha)
-ye^{-\pi y(2\alpha-\frac{1}{2\alpha})}\mathcal{P}(y;2\alpha)\\
&\geq\frac{\pi y}{2\alpha}-\frac{1}{2}-(\pi-\frac{1}{2}) e^{-\pi\frac{y}{2\alpha}}
-4.5e^{-\pi y(\alpha-\frac{1}{2\alpha})}
-10.5e^{-\pi y(2\alpha-\frac{1}{2\alpha})}\\
&=\frac{\pi}{2}\frac{y}{\alpha}-\frac{1}{2}-(\pi-\frac{1}{2}) e^{-\frac{\pi}{2}\frac{y}{\alpha}}
-4.5e^{-\alpha^2\cdot\pi \frac{y}{\alpha}(1-\frac{1}{2\alpha^2})}
-10.5e^{-\alpha^2\cdot\pi \frac{y}{\alpha}(2-\frac{1}{2\alpha^2})}\\
&\geq\frac{\pi}{2}\frac{y}{\alpha}-\frac{1}{2}-(\pi+4) e^{-\frac{\pi}{2}\frac{y}{\alpha}}
-10.5e^{-\frac{3\pi}{2}\frac{y}{\alpha}},
\endaligned\end{equation}
where $\alpha\geq1$ is used.
Next, a simple calculation shows that
\begin{equation}\aligned\nonumber
\frac{\pi}{2}\cdot x -\frac{1}{2}-(\pi+4) e^{-\frac{\pi}{2}\cdot x}
-10.5e^{-\frac{3\pi}{2}\cdot x}>0\Leftrightarrow x>1.126371\cdots.
\endaligned\end{equation}
Then by \eqref{qqq111},
\begin{equation}\aligned\label{qqq111}
\mathcal{Q}(y;\alpha)
&\geq\frac{\pi}{2}\frac{y}{\alpha}-\frac{1}{2}-(\pi+4) e^{-\frac{\pi}{2}\frac{y}{\alpha}}
-10.5e^{-\frac{3\pi}{2}\frac{y}{\alpha}}\\
&>0\;\;\;\;\hbox{if}\;\;\;\; \frac{y}{\alpha}>1.126371\cdots,
\endaligned\end{equation}
yields the result.

\end{proof}

\subsection{The estimates of $(\frac{\partial^2}{\partial y^2}+\frac{2}{y}\frac{\partial}{\partial y})\Big(\theta(\alpha;\frac{1}{2}+i y)-\sqrt2\theta(\alpha;\frac{1}{2}+i y)\Big)$}

The following Lemma is a particular case of Lemma \ref{Lemma47}, our analysis relies on this expression.
\begin{lemma}\label{Lemma413} The identity for $\big(\frac{\partial^2}{\partial y^2}+\frac{2}{y}\frac{\partial}{\partial y}\big)\big(\theta(\alpha;\frac{1}{2}+i y)-\sqrt2\theta(2\alpha;\frac{1}{2}+i y)\big)$ holds
 \begin{equation}\nonumber\aligned
\big(\frac{\partial^2}{\partial y^2}+\frac{2}{y}\frac{\partial}{\partial y}\big)&\big(\theta(\alpha;\frac{1}{2}+i y)-\sqrt2\theta(2\alpha;\frac{1}{2}+i y)\big)=(\pi\alpha)^2\sum_{n,m} (n^2-\frac{(m+\frac{n}{2})^2}{y^2})^2e^{-\pi\alpha(yn^2+\frac{(m+\frac{n}{2})^2}{y})}\\
&+\frac{4\sqrt2\pi\alpha}{y}\sum_{n,m}n^2e^{-2\pi\alpha(yn^2+\frac{(m+\frac{n}{2})^2}{y})}
-\frac{2\pi\alpha}{y}\sum_{n,m}n^2e^{-\pi\alpha(yn^2+\frac{(m+\frac{n}{2})^2}{y})}\\
&-4\sqrt2(\pi\alpha)^2\sum_{n,m} (n^2-\frac{(m+\frac{n}{2})^2}{y^2})^2e^{-2\pi\alpha(yn^2+\frac{(m+\frac{n}{2})^2}{y})}.
\endaligned\end{equation}

\end{lemma}
There are two types of double sums appeared in Lemma \ref{Lemma413}$($with slightly different frequencies$)$, as follows
\begin{equation}\aligned\label{gghh}
\hbox{double sum A}:&=\sum_{n,m}n^2e^{-\pi\alpha(yn^2+\frac{(m+\frac{n}{2})^2}{y})}, \\
\hbox{double sum B}:&=\sum_{n,m} (n^2-\frac{(m+\frac{n}{2})^2}{y^2})^2e^{-2\pi\alpha(yn^2+\frac{(m+\frac{n}{2})^2}{y})}.
\endaligned\end{equation}

We shall estimate these two double sums of \eqref{gghh} in Lemmas \ref{Lemma414} and \ref{Lemma415}.

\begin{lemma}\label{Lemma414} We have the following  upper bound function of $\sum_{n,m}n^2 e^{-\pi\alpha(yn^2+\frac{(m+\frac{n}{2})^2}{y})}$:
\begin{equation}\aligned\nonumber
\sum_{n,m}n^2 e^{-\pi\alpha(yn^2+\frac{(m+\frac{n}{2})^2}{y})}\leq4e^{-\pi\alpha(y+\frac{1}{4y})}\cdot\big(1+\epsilon_a\big),
\endaligned\end{equation}
where $\epsilon_a$ is small and can be explicitly controlled by
$$
\epsilon_a:=\epsilon_{a,1}+\epsilon_{a,2}+\epsilon_{a,3}+\epsilon_{a,4}
$$
and
\begin{equation}\aligned\nonumber
\epsilon_a\rightarrow0\;\;\hbox{as}\;\;y\mapsto\infty.
\endaligned\end{equation}

Here each $\epsilon_{a,j}(j=1,2,3,4)$ is small and expressed by
\begin{equation}\aligned\nonumber
\epsilon_{a,1}:&=\sum_{n=2}^\infty(2n-1)^2 e^{-\pi\alpha y((2n-1)^2-1)}\\
\epsilon_{a,2}:&=\sum_{n=2}^\infty e^{-\frac{\pi\alpha}{4y}((2n-1)^2-1)}\\
\epsilon_{a,3}:&=\epsilon_{a,1}\cdot\epsilon_{a,2}\\
\epsilon_{a,4}:&=2 e^{-\pi\alpha(3y-\frac{1}{4y})}\big(1+\sum_{n=2}^\infty n^2 e^{-4\pi\alpha y(n^2-1)}\big)\cdot\vartheta_3(\frac{\alpha}{y})
.
\endaligned\end{equation}

\end{lemma}

\begin{proof}
We shall divide the sum into two parts,
\begin{equation}\aligned\nonumber
\sum_{n,m}n^2 e^{-\pi\alpha(yn^2+\frac{(m+\frac{n}{2})^2}{y})}
&=\sum_{p,q,p\equiv q(mod2)}p^2 e^{-\pi\alpha(yp^2+\frac{q^2}{4y})}\\
&=\sum_{p\equiv q\equiv0(mod2)}p^2 e^{-\pi\alpha(yp^2+\frac{q^2}{4y})}+\sum_{p\equiv q\equiv1(mod2)}p^2 e^{-\pi\alpha(yp^2+\frac{q^2}{4y})}.
\endaligned\end{equation}
For convenience, we denote that
\begin{equation}\aligned\nonumber
J_1:&=\sum_{p\equiv q\equiv0(mod2)}p^2 e^{-\pi\alpha(yp^2+\frac{q^2}{4y})},\\
 J_2:&=\sum_{p\equiv q\equiv1(mod2)}p^2 e^{-\pi\alpha(yp^2+\frac{q^2}{4y})}.
\endaligned\end{equation}
Then
\begin{equation}\aligned\label{J1J2g}
\sum_{n,m}n^2 e^{-\pi\alpha(yn^2+\frac{(m+\frac{n}{2})^2}{y})}
=J_1+J_2.
\endaligned\end{equation}

We  now estimate $J_1$ and $J_2$ respectively. First $J_1$ can be rewritten as
\begin{equation}\aligned\label{J1g}
J_1=&\sum_{p\equiv q\equiv0(mod2)}p^2 e^{-\pi\alpha(yp^2+\frac{q^2}{4y})}
=\sum_{p=2n,q=2m}p^2 e^{-\pi\alpha(yp^2+\frac{q^2}{4y})}\\
=&4\sum_{n}n^2 e^{-4\pi\alpha y n^2}\sum_m e^{-\pi\frac{\alpha}{y}m^2}
=8\sum_{n=1}^\infty n^2 e^{-4\pi\alpha yn^2}\cdot(1+2\sum_{m=1}^\infty e^{-\pi\frac{\alpha}{y}m^2})\\
=& 8e^{-4\pi\alpha y}(1+\sum_{n=2}^\infty n^2 e^{-4\pi\alpha y(n^2-1)})\cdot(1+2\sum_{m=1}^\infty e^{-\pi\frac{\alpha}{y}m^2})\\
=&4 e^{-\pi\alpha(y+\frac{1}{4y})}\cdot \epsilon_{a,4},
\endaligned\end{equation}
as we can see later, $J_1$ is the remainder terms.

Next $J_2$ can be deformed as
\begin{equation}\aligned\label{J2g}
J_2=&\sum_{p\equiv q\equiv1(mod2)}p^2 e^{-\pi\alpha(yp^2+\frac{q^2}{4y})}
=\sum_{p=2n-1,q=2m-1}p^2 e^{-\pi\alpha(yp^2+\frac{q^2}{4y})}\\
=&4\sum_{n=1}^\infty (2n-1)^2 e^{-\pi\alpha y(2n-1)^2}\cdot\sum_{m=1}^\infty e^{-\pi\alpha\frac{(2m-1)^2}{4y}}\\
=&4 e^{-\pi\alpha(y+\frac{1}{4y})}\cdot(1+\sum_{n=2}^\infty (2n-1)^2 e^{-\pi\alpha y((2n-1)^2-1)})
\cdot(1+\sum_{m=2}^\infty e^{-\frac{\pi\alpha}{4y}((2m-1)^2-1)})\\
=&4 e^{-\pi\alpha(y+\frac{1}{4y})}\cdot(1+\epsilon_{a,1}+\epsilon_{a,2}+\epsilon_{a,1}\cdot\epsilon_{a,2})\\
=&4 e^{-\pi\alpha(y+\frac{1}{4y})}\cdot(1+\epsilon_{a,1}+\epsilon_{a,2}+\epsilon_{a,3}).
\endaligned\end{equation}

The result follows by \eqref{J1J2g}, \eqref{J1g} and \eqref{J2g}.

\end{proof}

\begin{lemma}\label{Lemma415} We have the following  upper bound
\begin{equation}\aligned\nonumber
\sum_{n,m}(n^2-\frac{(m+\frac{n}{2})^2}{y^2})^2 e^{-2\pi\alpha (yn^2+\frac{(m+\frac{n}{2})^2}{y})}\leq
\frac{2}{y^4}e^{-2\pi\frac{\alpha}{y}}\cdot(1+\epsilon_b),
\endaligned\end{equation}
where $\epsilon_b$ is small and consist of four smaller parts
\begin{equation}\aligned\nonumber
\epsilon_b:=\epsilon_{b,1}+\epsilon_{b,2}+\epsilon_{b,3}+\epsilon_{b,4},
\endaligned\end{equation}
and
\begin{equation}\aligned\nonumber
\epsilon_b\rightarrow0\;\;\hbox{as}\;\;y\mapsto\infty.
\endaligned\end{equation}

Here
\begin{equation}\aligned\nonumber
\epsilon_{b,1}
:&=2y^4 e^{-2\pi\alpha y}\cdot(1+\sum_{n=2}^\infty e^{-\frac{2\pi\alpha}{y}((2n-1)^2-1)})\cdot
(1+\sum_{n=2}^\infty (2n-1)^4 e^{-2\pi\alpha y ((2n-1)^2-1)})\\
\epsilon_{b,2}
:&=\frac{1}{8}e^{-2\pi\alpha y}\cdot(1+\sum_{n=2}^\infty (2n-1)^4 e^{-\frac{2\pi\alpha}{y}((2n-1)^2-1)})
\cdot(1+\sum_{n=2}^\infty e^{-2\pi\alpha y((2n-1)^2-1)})\\
\epsilon_{b,3}
:&=16y^4 e^{-\pi \alpha (8y-\frac{2}{y})}\cdot(1+\sum_{n=2}^\infty n^4e^{-8\pi\alpha y(n^2-1)})\cdot(1+2\sum_{n=1}^\infty e^{-2\pi\frac{\alpha}{y}n^2})\\
\epsilon_{b,4}
:&=y^4 e^{-\pi \alpha (8y-\frac{2}{y})}\cdot(1+\sum_{n=2}^\infty e^{-8\pi\alpha y(n^2-1)})\cdot(1+2\sum_{n=1}^\infty \frac{n^4}{y^4}e^{-2\pi\frac{\alpha}{y}n^2}).
\endaligned\end{equation}

\end{lemma}

\begin{proof}
We shall divide the sum into two different parts as follows
\begin{equation}\aligned\nonumber
&\sum_{n,m}(n^2-\frac{(m+\frac{n}{2})^2}{y^2})^2 e^{-2\pi\alpha (yn^2+\frac{(m+\frac{n}{2})^2}{y})}\\
&=\sum_{p,q,p\equiv q(mod2)}(p^2-\frac{q^2}{y^2})^2 e^{-2\pi\alpha (yp^2+\frac{q^2}{y})}\\
&=\sum_{p,q,p\equiv q\equiv0(mod2)}(p^2-\frac{q^2}{y^2})^2 e^{-2\pi\alpha (yp^2+\frac{q^2}{y})}\\
&+\sum_{p,q,p\equiv q\equiv1(mod2)}(p^2-\frac{q^2}{y^2})^2 e^{-2\pi\alpha (yp^2+\frac{q^2}{y})}.
\endaligned\end{equation}

For convenience, one denotes that
\begin{equation}\aligned\nonumber
\mathcal{K}_a:&=\sum_{p,q,p\equiv q\equiv0(mod2)}(p^2-\frac{q^2}{y^2})^2 e^{-2\pi\alpha (yp^2+\frac{q^2}{y})},\\
\mathcal{K}_b:&=\sum_{p,q,p\equiv q\equiv1(mod2)}(p^2-\frac{q^2}{y^2})^2 e^{-2\pi\alpha (yp^2+\frac{q^2}{y})}.
\endaligned\end{equation}
Hence we have
\begin{equation}\aligned\nonumber
\sum_{n,m}(n^2-\frac{(m+\frac{n}{2})^2}{y^2})^2 e^{-2\pi\alpha (yn^2+\frac{(m+\frac{n}{2})^2}{y})}
=\mathcal{K}_a+\mathcal{K}_b.
\endaligned\end{equation}
One deforms $\mathcal{K}_a$ and $\mathcal{K}_b$ respectively:
\begin{equation}\aligned\nonumber
\mathcal{K}_a&=\sum_{p,q,p\equiv q\equiv0(mod2)}(p^2-\frac{q^2}{y^2})^2 e^{-2\pi\alpha (yp^2+\frac{q^2}{y})}\\
&=\sum_{n,m}(4n^2-\frac{m^2}{y^2})^2 e^{-2\pi\alpha(4n^2 y+\frac{m^2}{y})}
\endaligned\end{equation}
To leading order, we single out the major terms by regrouping the terms as follows
\begin{equation}\aligned\nonumber
\mathcal{K}_a
&=\sum_{n,m}(4n^2-\frac{m^2}{y^2})^2 e^{-2\pi\alpha(4n^2 y+\frac{m^2}{y})}\\
&=\frac{2}{y^4}e^{-2\pi\frac{\alpha}{y}}
+\sum_{n\not\in\{0\},m}(4n^2-\frac{m^2}{y^2})^2 e^{-2\pi\alpha(4n^2 y+\frac{m^2}{y})}\\
&\;\;\;\;\;\;\;\;\;\;\;\;\;\;\;\;\;+\sum_{m\not\in\{-1,1\},n}(4n^2-\frac{m^2}{y^2})^2 e^{-2\pi\alpha(4n^2 y+\frac{m^2}{y})}.
\endaligned\end{equation}
To further simplify the structure, one denotes that
\begin{equation}\aligned\nonumber
\mathcal{K}_{a,1}:&=\sum_{n\not\in\{0\},m}(4n^2-\frac{m^2}{y^2})^2 e^{-2\pi\alpha(4n^2 y+\frac{m^2}{y})},\\
\mathcal{K}_{a,2}:&=\sum_{m\not\in\{-1,1\},n}(4n^2-\frac{m^2}{y^2})^2 e^{-2\pi\alpha(4n^2 y+\frac{m^2}{y})}.
\endaligned\end{equation}
Then,
\begin{equation}\aligned\label{Ka}
\mathcal{K}_a=
\frac{2}{y^4}e^{-2\pi\frac{\alpha}{y}}
+\mathcal{K}_{a,1}+\mathcal{K}_{a,2}
\endaligned\end{equation}
and
\begin{equation}\aligned\label{minus}
\sum_{n,m}(n^2-\frac{(m+\frac{n}{2})^2}{y^2})^2 e^{-2\pi\alpha (yn^2+\frac{(m+\frac{n}{2})^2}{y})}
=\frac{2}{y^4}e^{-2\pi\frac{\alpha}{y}}
+\mathcal{K}_{a,1}+\mathcal{K}_{a,2}+\mathcal{K}_b.
\endaligned\end{equation}
To control $\mathcal{K}_{a,j}, j=1,2$, we use a basic mean value inequality and
\begin{equation}\aligned\nonumber
\mathcal{K}_{a,1}&=\sum_{n\not\in\{0\},m}(4n^2-\frac{m^2}{y^2})^2 e^{-2\pi\alpha(4n^2 y+\frac{m^2}{y})}\\
&\leq\sum_{n\not\in\{0\},m}16n^4e^{-2\pi\alpha(4n^2 y+\frac{m^2}{y})}+\sum_{n\not\in\{0\},m}\frac{m^4}{y^4}e^{-2\pi\alpha(4n^2 y+\frac{m^2}{y})}\\
&=32\sum_{n=1}^\infty n^4e^{-8\pi\alpha y n^2}\sum_m e^{-2\pi\frac{\alpha}{y}m^2}
+2\sum_{n=1}^\infty e^{-8\pi\alpha y n^2}\sum_m \frac{m^4}{y^4}e^{-2\pi\frac{\alpha}{y}m^2}\\
&=\frac{2}{y^4}e^{-2\pi\frac{\alpha}{y}}\cdot \sigma_{\mathcal{K}_1},
\endaligned\end{equation}
where we single  out the small remainder terms denoted by $\sigma_{\mathcal{K}}$ as follows
\begin{equation}\aligned\label{K1}
 \sigma_{\mathcal{K}_1}
 :&=
 16y^4 e^{-\pi\alpha(8y-\frac{2}{y})}\cdot (1+\sum_{n=2}^\infty n^4 e^{-8\pi\alpha y(n^2-1)})
 \cdot (1+2\sum_{n=1}^\infty e^{-2\pi\frac{\alpha}{y}n^2})\\
 &+ e^{-\pi\alpha (8y-\frac{2}{y})}
 \cdot (1+\sum_{n=2}^\infty e^{-8\pi\alpha y(n^2-1)})
 \cdot (1+2\sum_{n=1}^\infty n^4e^{-2\pi\frac{\alpha}{y}n^2})
\endaligned\end{equation}
Similar to $\mathcal{K}_{a,2}$,
\begin{equation}\aligned\nonumber
\mathcal{K}_{a,2}&=\sum_{m\not\in\{-1,1\},n}(4n^2-\frac{m^2}{y^2})^2 e^{-2\pi\alpha(4n^2 y+\frac{m^2}{y})}\\
&\leq\sum_{m\not\in\{-1,1\},n} 16n^4 e^{-2\pi\alpha(4n^2 y+\frac{m^2}{y})}
+\sum_{m\not\in\{-1,1\},n} \frac{m^4}{y^4} e^{-2\pi\alpha(4n^2 y+\frac{m^2}{y})}\\
&\leq\sum_{m,n} 16n^4 e^{-2\pi\alpha(4n^2 y+\frac{m^2}{y})}
+\sum_{m\not\in\{-1,1\},n} \frac{m^4}{y^4} e^{-2\pi\alpha(4n^2 y+\frac{m^2}{y})}\\
&\leq32\sum_{n=1}^\infty n^4 e^{-8\pi\alpha y n^2}\sum_{m} e^{-2\pi\frac{\alpha}{y}m^2}
+2\sum_{m=2}^\infty \frac{m^4}{y^4} e^{-2\pi\frac{\alpha}{y}m^2}\sum_{n} e^{-8\pi\alpha y n^2}\\
&=\frac{2}{y^4}e^{-2\pi\frac{\alpha}{y}}\cdot \sigma_{\mathcal{K}_2},
\endaligned\end{equation}
where
\begin{equation}\aligned\label{K2}
\sigma_{\mathcal{K}_2}&=16e^{-\pi\alpha (8y-\frac{2}{y})}\cdot(1+\sum_{n=2}^\infty n^4 e^{-8\pi\alpha y (n^2-1)})\cdot
(1+2\sum_{m=1}^\infty e^{-2\pi\frac{\alpha}{y}m^2})\\
&\;+\sum_{m=2}^\infty m^4 e^{-2\pi\frac{\alpha}{y}(m^2-1)}\cdot(1+2\sum_{m=1}^\infty e^{-8\pi\alpha y n^2}).
\endaligned\end{equation}
Recall in \eqref{Ka}, one has
\begin{equation}\aligned\label{kka}
\mathcal{K}_a&=
\frac{2}{y^4}e^{-2\pi\frac{\alpha}{y}}
+\mathcal{K}_{a,1}+\mathcal{K}_{a,2}
&=\frac{2}{y^4}e^{-2\pi\frac{\alpha}{y}}\cdot(1+\sigma_{\mathcal{K}_1}+\sigma_{\mathcal{K}_2}),
\endaligned\end{equation}
where $\sigma_{\mathcal{K}_1}$ and $\sigma_{\mathcal{K}_2}$ are defined in \eqref{K1} and \eqref{K2} respectively.

Next, we estimate $\mathcal{K}_b$.
\begin{equation}\aligned\label{kkb}
\mathcal{K}_b&=\sum_{p,q,p\equiv q\equiv1(mod2)}(p^2-\frac{q^2}{y^2})^2 e^{-2\pi\alpha (yp^2+\frac{q^2}{y})}\\
&=\sum_{n,m}((2n-1)^2-\frac{(2m-1)^2}{4y^2})^2e^{-2\pi\alpha (y(2n-1)^2+\frac{(2m-1)^2}{y})}\\
&\leq\sum_{n,m}(2n-1)^4e^{-2\pi\alpha (y(2n-1)^2+\frac{(2m-1)^2}{y})}
+\sum_{n,m}\frac{(2m-1)^4}{16y^4}e^{-2\pi\alpha (y(2n-1)^2+\frac{(2m-1)^2}{y})}\\
&=\sum_m e^{-2\pi\frac{\alpha}{y}(2m-1)^2}\cdot\sum_n (2n-1)^4e^{-2\pi\alpha y(2n-1)^2}\\
&\;\;\;+\sum_m \frac{(2m-1)^4}{16y^4}e^{-2\pi\frac{\alpha}{y}(2m-1)^2}\cdot\sum_n e^{-2\pi\alpha y(2n-1)^2}\\
&=\frac{2}{y^4}e^{-2\pi\frac{\alpha}{y}}\cdot \sigma_{\mathcal{K}_3},
\endaligned\end{equation}
where
\begin{equation}\aligned\label{K3}
 \sigma_{\mathcal{K}_3}:=2y^4 e^{-2\pi\alpha y}\cdot(1+\sum_{n=2}^\infty e^{-2\pi\frac{\alpha}{y}((2n-1)^2-1)})
 \cdot(1+\sum_{n=2}^\infty (2n-1)^4 e^{-2\pi\alpha y((2n-1)^2-1)})\\
 +\frac{1}{8}e^{-2\pi \alpha y}\cdot(1+\sum_{n=2}^\infty (2n-1)^4 e^{-2\pi\frac{\alpha}{y}((2n-1)^2-1)})
 \cdot(1+\sum_{n=2}^\infty  e^{-2\pi \alpha y ((2n-1)^2-1)}).
\endaligned\end{equation}
Combining \eqref{minus} with \eqref{kka} and \eqref{kkb}, one deduces that
\begin{equation}\aligned\label{minusaaa}
\sum_{n,m}(n^2-\frac{(m+\frac{n}{2})^2}{y^2})^2 e^{-2\pi\alpha (yn^2+\frac{(m+\frac{n}{2})^2}{y})}
&=\frac{2}{y^4}e^{-2\pi\frac{\alpha}{y}}
+\mathcal{K}_{a,1}+\mathcal{K}_{a,2}+\mathcal{K}_b\\
&\leq\frac{2}{y^4}e^{-2\pi\frac{\alpha}{y}}\cdot(1+\sigma_{\mathcal{K}_1}+\sigma_{\mathcal{K}_2}+\sigma_{\mathcal{K}_3}),
\endaligned\end{equation}
where $\sigma_{\mathcal{K}_1}$, $\sigma_{\mathcal{K}_2}$ and $\sigma_{\mathcal{K}_3}$ are defined in
\eqref{K1}, \eqref{K2} and \eqref{K3} respectively. The inequality \eqref{minusaaa} yields the result.

\end{proof}

The next two Lemmas provide the lower bound functions of the double sums in Lemma \ref{Lemma413}, where the positiveness is used effectively.

\begin{lemma}\label{Lemma416} A lower bound function of the double sum $\sum_{n,m} (n^2-\frac{(m+\frac{n}{2})^2}{y^2})^2e^{-\pi\alpha(yn^2+\frac{(m+\frac{n}{2})^2}{y})}$ is as follows
 \begin{equation}\nonumber\aligned
\sum_{n,m} (n^2-\frac{(m+\frac{n}{2})^2}{y^2})^2e^{-\pi\alpha(yn^2+\frac{(m+\frac{n}{2})^2}{y})}
\geq\frac{2}{y^4} e^{-\pi\frac{\alpha}{y}}+4(1-\frac{1}{4y^2})^2 e^{-\pi\alpha(y+\frac{1}{4y})}.
\endaligned\end{equation}

\end{lemma}
\begin{remark} In the proof of Lemma \ref{Lemma416}$($and Lemma \ref{Lemma417} below$)$, we have used the positive structure of the double sum.
\end{remark}
\begin{proof} The double sum evaluates at
$$(m,n)=\{(1,0),(-1,0)\}
\;\;\hbox{contributing}\;\;\frac{1}{y^4} e^{-\pi\frac{\alpha}{y}}\;\;\hbox{each}
$$
and
$$(m,n)=\{(0,1),(0,-1),(1,-1),(-1,1)\}
\;\;\hbox{contributing}\;\;(1-\frac{1}{4y^2})^2 e^{-\pi\alpha(y+\frac{1}{4y})}\;\;\hbox{each}.
$$
The rest of other terms in the double sum all are positive and hence the result follows.
\end{proof}

\begin{lemma}\label{Lemma417}
A lower bound function of $\sum_{n,m}n^2e^{-2\pi\alpha(yn^2+\frac{(m+\frac{n}{2})^2}{y})}$ is
\begin{equation}\nonumber\aligned
\sum_{n,m}n^2e^{-2\pi\alpha(yn^2+\frac{(m+\frac{n}{2})^2}{y})}
\geq4 e^{-2\pi\alpha (y+\frac{1}{4y})}.
\endaligned\end{equation}

\end{lemma}
\begin{proof} The double sum can be evaluated at
$$(m,n)=\{(0,1),(0,-1),(1,-1),(-1,1)\}
\;\;\hbox{contributing}\;\; e^{-2\pi\alpha(y+\frac{1}{4y})}\;\;\hbox{each}.
$$
The rest of other terms in the double sum all are positive and hence the result follows.

\end{proof}

\begin{lemma}\label{Lemma418} We have the following   lower bound estimate

\begin{equation}\aligned\nonumber
\big(\frac{\partial^2}{\partial y^2}+\frac{2}{y}\frac{\partial}{\partial y}\big)(\theta(\alpha;\frac{1}{2}+iy)-\sqrt2\theta(2\alpha;\frac{1}{2}+iy))
\geq\frac{2(\pi\alpha)^2}{y^4}e^{-\pi\frac{\alpha}{y}}\cdot
\mathcal{W}(y;\alpha),
\endaligned\end{equation}
where
\begin{equation}\aligned\label{www111}
\mathcal{W}(y;\alpha):=1+(2(y^2-\frac{1}{4})^2-\frac{4}{\pi\alpha}y^3\cdot(1+\epsilon_{a}))\cdot e^{-\pi\alpha (y-\frac{3}{4y})}
-4\sqrt2(1+\epsilon_{b})\cdot e^{-\pi\frac{\alpha}{y}}.
\endaligned\end{equation}
Here $\epsilon_{a}$ and $\epsilon_{b}$ are defined in Lemmas \ref{Lemma414} and \ref{Lemma415} respectively.
\end{lemma}

\begin{proof} In view of Lemma \ref{Lemma413}, combining with the bound functions in Lemmas \ref{Lemma414}-\ref{Lemma418}, we have
 \begin{equation}\nonumber\aligned
\big(\frac{\partial^2}{\partial y^2}+\frac{2}{y}\frac{\partial}{\partial y}\big)&\big(\theta(\alpha;\frac{1}{2}+i y)-\sqrt2\theta(2\alpha;\frac{1}{2}+i y)\big)=(\pi\alpha)^2\sum_{n,m} (n^2-\frac{(m+\frac{n}{2})^2}{y^2})^2e^{-\pi\alpha(yn^2+\frac{(m+\frac{n}{2})^2}{y})}\\
&+\frac{4\sqrt2\pi\alpha}{y}\sum_{n,m}n^2e^{-2\pi\alpha(yn^2+\frac{(m+\frac{n}{2})^2}{y})}
-\frac{2\pi\alpha}{y}\sum_{n,m}n^2e^{-\pi\alpha(yn^2+\frac{(m+\frac{n}{2})^2}{y})}\\
&-4\sqrt2(\pi\alpha)^2\sum_{n,m} (n^2-\frac{(m+\frac{n}{2})^2}{y^2})^2e^{-2\pi\alpha(yn^2+\frac{(m+\frac{n}{2})^2}{y})}\\
\geq&(\pi\alpha)^2\frac{2}{y^4}e^{-\pi\frac{\alpha}{y}}+4(\pi\alpha)^2(1-\frac{1}{4y^2})^2 e^{-\pi\alpha(y+\frac{1}{4y})}
+\frac{16\sqrt2\pi\alpha}{y}e^{-2\pi\alpha(y+\frac{1}{4y})}\\
&\;\;-\frac{8\pi\alpha}{y}(1+\epsilon_a)e^{-\pi\alpha(y+\frac{1}{4y})}
-\frac{8\sqrt2(\pi\alpha)^2}{y^4}(1+\epsilon_b) e^{-2\pi\frac{\alpha}{y}}\\
\geq&\frac{2(\pi\alpha)^2}{y^4}e^{-\pi\frac{\alpha}{y}}\cdot
\mathcal{W}(y;\alpha).
\endaligned\end{equation}

\end{proof}

\begin{lemma}\label{Lemma420} Assume that $\alpha\geq1, y\geq\frac{\sqrt3}{2}$. If $\frac{y}{\alpha}\leq3$, then
\begin{equation}\nonumber\aligned
\epsilon_a&=0.1264717\cdots<0.15,\\
\epsilon_b&=0.0054169\cdots<0.006.
\endaligned\end{equation}
\end{lemma}
\begin{proof} The terms of $\epsilon_a, \epsilon_b$ are exponentially decaying. One needs to use the fact that
$\alpha\geq1, y\geq\frac{\sqrt3}{2}$ and $\frac{\alpha}{y}\geq\frac{1}{3}$. Note that the positive lower bounds of $\frac{\alpha}{y}$ and $\alpha y$ are used effectively to control the summation.

\end{proof}

\begin{lemma}\label{Lemma421} A refined lower bound of $\mathcal{W}(y;\alpha)$, which is defined in \eqref{www111}, is the following
\begin{equation}\aligned\nonumber
\mathcal{W}(y;\alpha)\geq
\begin{cases}
1-(\frac{4(1+\epsilon_a)}{\pi}-\frac{9}{8})-4\sqrt2(1+\epsilon_b) e^{-\pi},\;\hbox{if}\; y\in[\frac{\sqrt3}{2},1];\\
1-(\frac{4(1+\epsilon_a)}{\pi}y^3-2(y^2-\frac{1}{4})^2)\cdot e^{-\frac{\pi}{4}}-4\sqrt2(1+\epsilon_b) e^{-\frac{1}{y_\epsilon}\pi},
\;\hbox{if}\; y\in[1,y_\epsilon],\\
1-4\sqrt2(1+\epsilon_b) e^{-\pi\frac{\alpha}{y}},\;\;\hbox{if}\; y\in[y_\epsilon,\infty).
\end{cases}
\endaligned\end{equation}
Here $y_\epsilon$ is the unique root of
$$2(y^2-\frac{1}{4})^2-\frac{4}{\pi}y^3\cdot(1+\epsilon_{a})=0$$
on $[\frac{\sqrt3}{2},\infty)$. Numerically, $y_\epsilon\cong1.130998\cdots$.

\end{lemma}

\begin{proof} The proof is based on the explicit expression of $\mathcal{W}(y;\alpha)$ in Lemma \ref{Lemma418}. Each part of
$\mathcal{W}(y;\alpha)$ is analyzed separately.

\end{proof}

\begin{lemma} \label{Lemma422} Assume that $\alpha\geq1$. If $y\in[\frac{\sqrt3}{2},1.8\alpha]$, then
$$\mathcal{W}(y;\alpha)>0.$$
\end{lemma}

\begin{proof} The proof follows from Lemma \ref{Lemma421}. Note that
\begin{equation}\aligned\nonumber
1-(\frac{4(1+\epsilon_a)}{\pi}-\frac{9}{8})-4\sqrt2(1+\epsilon_b) e^{-\pi}>0.414852\cdots>0\\
1-(\frac{4(1+\epsilon_a)}{\pi}y^3-2(y^2-\frac{1}{4})^2)\cdot e^{-\frac{\pi}{4}}-4\sqrt2(1+\epsilon_b) e^{-\frac{1}{y_\epsilon}\pi}>0.491478\cdots,
\;\hbox{if}\; y\in[1,y_\epsilon],
\endaligned\end{equation}
and that
\begin{equation}\nonumber\aligned
1-4\sqrt2(1+\epsilon_b) e^{-\pi\frac{\alpha}{y}}>0\;\;\hbox{if}\;\;\frac{\alpha}{y}>0.553493\cdots.
\endaligned\end{equation}
Next, a simple computation shows that
\begin{equation}\nonumber\aligned
\frac{\alpha}{y}>0.553493\cdots\Leftrightarrow y\leq(1.806707\cdots)\alpha.
\endaligned\end{equation}
In view of Lemma \ref{Lemma421}, the result then follows.

\end{proof}

\section{Proofs of Theorems \ref{Th2}-\ref{ThA} }
\setcounter{equation}{0}

\noindent
{\bf Proof of Theorem \ref{Th2}:} By Fourier transform, we have 
\begin{equation}\aligned\nonumber
\theta(\frac{1}{\alpha};z)=\alpha\cdot\theta({\alpha};z), \;\;\alpha>0.
\endaligned\end{equation}
Theorem \ref{Th2} is equivalent to Theorem \ref{Th1}.

\medskip

\noindent
{\bf Proof of Theorem \ref{Th3}-\ref{Th4}:} These two theorems are easy consequences of Theorems \ref{Th1}-\ref{Th2}.

\medskip

\noindent
{\bf Proof of Theorem \ref{ThA}:}  The proof is based on an effective iteration scheme.

\noindent
{\bf Case A: $\beta\leq(\sqrt2)^k$.} We use the scheme
\begin{equation}\aligned\label{KKK}
&\Big(\theta (\alpha; z)-\beta\theta (2^k\alpha; z)\Big)\\
=&\big((\sqrt2)^k-\beta\big)\theta (2^k \alpha; z)
+\sum_{n=0}^{k-1}(\sqrt2)^n\Big(\theta (2^n\alpha; z)-\sqrt2\theta (2^{n+1}\alpha; z)\Big).
\endaligned\end{equation}
Note that all the coefficients in \eqref{KKK} are nonnegative. We apply Theorem \ref{Th1} on each term of \eqref{KKK} to arrive that the minimizer of $\Big(\theta (\alpha; z)-\beta\theta (2^k\alpha; z)\Big)$ is $\frac{1}{2}+i\frac{\sqrt3}{2}$ again in this case.

\medskip

\noindent
{\bf Case B: $\beta>(\sqrt2)^k$.}
The proof of nonexistence of the minimizer is similar to that of Lemma \ref{Lemma41}.

\bigskip

{\bf Acknowledgements.}
 S. Luo is grateful to Professor H.J. Zhao(Wuhan University) for his constant support and encouragement. The research of S. Luo is partially supported by double thousands plan of Jiangxi(jxsq2019101048) and NSFC(No. 12001253). The research of J. Wei is partially supported by NSERC of Canada.

\bigskip


\end{document}